\documentclass{article}

\usepackage{microtype}
\usepackage{graphicx}
\usepackage{booktabs} 

\usepackage{hyperref}
\usepackage{url}            
\usepackage{booktabs}       
\usepackage{amsfonts}       
\usepackage{nicefrac}       
\usepackage{microtype}      
\usepackage{xcolor}         
\usepackage{booktabs}

\usepackage{amsmath, amsfonts, amsthm, amssymb,enumitem}
\usepackage{color, bm, graphicx} 
\usepackage{booktabs, multicol, multirow}
\usepackage{xspace}
\usepackage{algorithm}
\usepackage{algorithmic}
\usepackage{verbatim}
\usepackage{amsmath}
\usepackage{amssymb}
\usepackage{mathtools}
\usepackage{amsthm}
\newcommand{\note}[1]{{\color{red}[1]}}
\usepackage{bm}
\usepackage{subcaption}

\usepackage{multirow}

\DeclareMathOperator*{\argmin}{arg\,min}

  \def\x{\bm{x}}
  
  \def\y{\bm{y}}

  \def\R{\mathbb{R}}

  \def\1{{\bf{1}}}

  \def\s{\bm{s}}

  \def\Z{\mathcal{Z}}
  \def\K{\mathcal{K}}
  \def\c{\bm{c}}
  
  \def\b{\bm{b}}

  \def\btheta{\bm{\theta}}
  \def\blambda{\bm{\lambda}}
  \def\z{\bm{z}}
  \def\bgamma{\bm{\gamma}}

  \def\0{\bm{0}}

\usepackage[accepted]{icml2025}

\theoremstyle{plain}
\newtheorem{theorem}{Theorem}[section]
\newtheorem{proposition}[theorem]{Proposition}
\newtheorem{lemma}[theorem]{Lemma}

\newtheorem{definition}[theorem]{Definition}

\newtheorem{remark}[theorem]{Remark}
\newtheorem{example}[theorem]{Example}
\newtheorem{observation}[theorem]{Observation}

\usepackage[textsize=tiny]{todonotes}

\icmltitlerunning{Inverse Optimization via Learning Feasible Regions}

\begin{document}

\twocolumn[
\icmltitle{Inverse Optimization via Learning Feasible Regions}




\begin{icmlauthorlist}
\icmlauthor{Ke Ren}{a}
\icmlauthor{Peyman Mohajerin Esfahani}{b}
\icmlauthor{Angelos Georghiou}{c}
\end{icmlauthorlist}

\icmlaffiliation{a}{Amazon}
\icmlaffiliation{b}{University of Toronto and Delft University of Technology}
\icmlaffiliation{c}{University of Cyprus}

\icmlcorrespondingauthor{Ke Ren}{renkea@amazon.com }

\vskip 0.3in
]



\printAffiliationsAndNotice{}  

\begin{abstract}
We study inverse optimization (IO) where the goal is to use a parametric optimization program as the hypothesis class to infer relationships between input-decision pairs. Most of the literature focuses on learning only the objective function, as learning the constraint function (i.e., feasible regions) leads to nonconvex training programs. Motivated by this, we focus on learning feasible regions for known linear objectives and introduce two training losses along with a hypothesis class to parameterize the constraint function. Our hypothesis class surpasses the previous objective-only method by naturally capturing discontinuous behaviors in input-decision pairs. We introduce a customized block coordinate descent algorithm with a smoothing technique to solve the training problems, while for 
further restricted hypothesis classes, we reformulate the training optimization as a tractable convex program or mixed integer linear program. Synthetic experiments and two power system applications, including comparisons with state-of-the-art approaches, showcase and validate the proposed approach.
\end{abstract}

\section{Introduction}
\label{submission}

Inverse optimization (IO) reverses traditional optimization. While classical optimization finds optimal decisions based on predefined objectives and constraints, IO takes decisions as inputs and identifies the objective and/or constraints that make these decisions either approximately or precisely optimal. Interest in IO has surged in recent years, leading to advancements in learning theory~\cite{aswani2018inverse, mohajerin2018data, chan2020inverse} and recently in reinforcement learning~\cite{ref:Amazon}, as well as applications such as transportation~\cite{zhang2017data, chen2021inverse}, system control~\cite{akhtar2021learning}, robotics~\cite{ref:offlineRL}, healthcare \cite{ayer2015inverse, ajayi2022objective}, and finance \cite{li2021inverse}. More recently, the IO framework has also been successfully applied to the Amazon Last Mile Routing Research Challenge, where the goal is to learn and replicate human drivers' routing preferences~\cite{ref:Amazon}. 

The IO model builds on a parametric \emph{forward} optimization model that represents the decision-generating process. Given an input signal $\bm s\in\mathbb{R}^k$, the following optimization problem defined through $f,g:\mathbb{R}^k\times\mathbb{R}^n\rightarrow \mathbb{R}$, generates an optimal solution $\bm x\in\mathbb{R}^n$ which is observed, however, it is possible that the observation is contaminated by noise. 
\begin{equation}\label{forward_problem}
    \begin{array} {cl}
        \displaystyle\min_{\bm x\in\mathbb{R}^n}
        & f(\bm x,\bm s) \qquad {\rm s.t.} \quad  g(\bm x,\bm s) \leq 0
    \end{array}
\end{equation}
In the most general case, decision makers have no knowledge of $f$ and $g$  but have access to $N$ independent pairs of input/decisions $\{\s_i,\x_i\}_{i = 1}^N$ from problem \eqref{forward_problem}.  As the space of all possible objective and constraint functions is vast, the decision-maker seeks to approximate $f$ and $g$ by candidates from some parametric hypothesis spaces $f_{\btheta}$ and $g_{\btheta}$, $\btheta \in \Theta$, where $\Theta$ represents a finite-dimensional parameter set. Correspondingly, in IO, decision-makers aim to learn $\btheta$ such that for each signal $\s_n$, $\x_n$ is optimal (or approximately optimal) in the following problem:
\begin{equation}\label{recovered_p}
   \begin{array}{cl}
        \displaystyle\min_{\bm x \in\mathbb{R}^n}
        & f_{\bm \theta}(\bm x,\bm s) \qquad {\rm s.t.} \quad  g_{\bm \theta}(\bm x,\bm s) \leq 0.
    \end{array}
\end{equation}
Through a supervised learning lens, the parameterized model~\eqref{recovered_p} can be viewed as a hypothesis class for learning the mapping between the input~$\bm s$ and the output decision~$\bm x$ generated by the forward optimization~\eqref{forward_problem}. 
 
\subsection{Contributions}
This paper studies the learning problem of IO where, unlike the majority of the literature that focuses on a parametric form for the objective function~$f_{\bm \theta}$, the main objective is to learn the constraints function~$g_{\bm \theta}$. In this context, the paper has three main contributions:  

    \noindent\textbf{Loss functions compatible with IO data:} Generalizing existing literature, we introduce two loss functions $\ell_{\bm \theta}(\bm x,\bm s)$ to evaluate model fit \eqref{recovered_p} with parameter $\bm \theta$ on the input-output data pair $(\bm x, \bm s)$ from \eqref{forward_problem} (Proposition \ref{prop:full}). Regarding performance on unseen (test) data, we argue that the IO setting encompasses a ``true" counterpart of these losses, unlike general supervised learning problems; see Figure~\ref{fig:Losses2} (bottom) for their geometric representation.
    
    \noindent \textbf{Hypothesis Classes with Convex and MILP Reformulations:} 
    We propose a new hypothesis class to parameterize the constraint function $ g_{\bm{\theta}}$, which is significantly richer than existing IO approaches (Example~\ref{example1}). However, it results in non-convex training problems. We demonstrate that special cases of this hypothesis class allow the optimization of the training loss $ \ell_{\bm{\theta}} $ to be exactly reformulated into tractable convex optimization (Theorem \ref{thm:convex}) and mixed-integer linear programs (MILP) (Proposition \ref{prop:milp}), which can be solved efficiently with off-the-shelf solvers.

    \noindent\textbf{Smoothed block-coordinate descent algorithm:} 
    We further exploit the structure of the proposed losses and devise a tailored block coordinate descent algorithm together with a smoothing technique to train our IO setting in for the generic hypothesis class~(Algorithm~\ref{Alg:smooth}). The smoothing technique provides an interesting insight into the relation between the two proposed loss functions, justifying why one often presents better computation results when optimized using vanilla gradient descent (Algorithm~\ref{alg:vanilla}); see Remark~\ref{rem:smoothness}.

The paper is organized as follows: Section~\ref{sec:loss} introduces two losses and explores their properties. Section~\ref{sec:hypothesis} presents the hypothesis class together with the proposed  coordinate descent algorithm with a smoothing technique to solve the general problem, along with convex and MILP reformulations for specific cases. Section~\ref{sec:real} covers computational studies. Limitations and future work are discussed in Section~\ref{sec:conclusion}. Proofs and additional computational studies are in the appendix.

\subsection{Related Works}
Extensive studies have been carried out to estimate the objective functions in data-driven IO. The various approaches primarily diverge in terms of the loss functions utilized to capture the disparity between predictions and observations. These include the KKT loss \cite{keshavarz2011imputing}, first-order loss \cite{bertsimas2015data}, predictability loss \cite{aswani2018inverse}, and suboptimality loss \cite{mohajerin2018data}. The properties and relationships of these approaches are summarized \cite{mohajerin2018data}. Other methodologies are also proposed under different settings including online settings \cite{barmann2017emulating, dong2018generalized}. In the landscape of IO for feasible regions, the majority of existing works have traditionally concentrated on right-hand side parameters \cite{dempe2006inverse, guler2010capacity, xu2016data, saez2016data, lu2018data}. However, recent endeavors by \cite{aswani2018inverse,tan2020learning,ghobadi2020inferring} have expanded the scope by considering various aspects of constraint parameters.

The work \cite{aswani2018inverse}  incorporates general constraint parameters into their model, departing from the mainstream.  However, their model encounters intractability issues, relying on enumerations to evaluate the problem. Similarly, \cite{tan2020learning} faces tractability challenges, employing sequential quadratic programming within a bi-level optimization framework for solving constraint parameters in linear programming. While \cite{chan2020inverse} and \cite{ghobadi2020inferring} propose tractable solutions, their settings impose some restrictions: \cite{chan2020inverse} does not consider a data-driven setting and confines their analysis to a single observation in linear programming. The paper~\cite{ghobadi2020inferring} extends this limitation to multi-point scenarios, but only for the linear programming case where the multiple points are merely feasible solutions, not necessarily optimal.

\section{Learning Frameworks}\label{sec:loss}

\textbf{Inverse optimization as supervised learning.}
From a machine learning perspective, the IO problem can be categorized as a supervised learning problem, where $\s$ represents the independent variable and $\x$ represents the response. Therefore, it is natural to define a loss minimization procedure (as shown in (\ref{saa_loss})) to find the unknown parameters $\btheta$.
\begin{equation}\label{saa_loss}
\begin{array}{ll}
        \displaystyle\min_{\bm \theta \in \Theta}&\displaystyle  \frac{1}{N}\sum_{i=1}^{N} \ell_{\bm \theta}(\bm x_i,\bm s_i).
\end{array}
\end{equation} 
An appropriate loss is required to measure both the feasibility and optimality of any given solution $\x$. 

\subsection{Loss Function}

In this section, we introduce two loss functions that relax feasibility and optimality in problem \eqref{recovered_p} while penalizing deviations. Two potential reasons why problem \eqref{recovered_p} might fail to replicate the feasibility and optimality of problem \eqref{forward_problem} are: $(i)$ the hypothesis class used for the constraints $g_{\bm \theta}$ and/or the objective $f_{\bm \theta}$ may lack the complexity to capture the behavior of \eqref{forward_problem}, and $(ii)$ the training data could be noisy, meaning the signal $\bm s$ and/or the decisions $\bm x$ may be influenced by unaccounted measurement noise. Typically, the degree of infeasibility of a data point can be quantified by $g_{\bm \theta}(\bm x,\bm s)$, with positive values indicating infeasibility, while suboptimality can be assessed via
\begin{equation}\label{function:suboptimality}
    J_{\bm\theta}(\bm x, \bm s):= f_{\bm \theta}(\bm x,\bm s) -  \left[
    \begin{array}{cl}
\displaystyle\min_{\bm y\in\mathbb{R}^n}  & f_{\bm \theta}(\bm y,\bm s)\\
{\rm s.t.}& g_{\bm \theta}(\bm y,\bm s) \leq 0
    \end{array}
    \right].
\end{equation}
For a fixed $\bm\theta$, $J_{\bm\theta}(\bm x, \bm s)$ quantifies the difference between the objective value achieved by data point $(\bm x, \bm s)$ under the hypothesized objective $f_{\bm \theta}(\bm x,\bm s)$ and the optimal value of \eqref{recovered_p} with signal $\bm s$. In noise-free scenarios, negative $J_{\bm\theta}(\bm x, \bm s)$ suggest over-constraint in \eqref{recovered_p}, while positive values imply overly relaxed feasibility, rendering $(\bm x, \bm s)$ suboptimal.

The loss functions are termed \emph{predictability} and \emph{suboptimality loss}, respectively.
\begin{subequations}\label{losses}
\begin{equation} 
 \ell_{\bm\theta}^\text{p}(\bm x,\bm s) := \left[
    \begin{array}{rl}
     \displaystyle\min  & \|\bm \gamma\|\\
   {\rm s.t.} & \bm \gamma\in\mathbb{R}^n\\
   &g_{\bm \theta} (\bm x+\bm \gamma,\bm s) \leq 0\\
    & J_{\bm\theta}(\bm x+\bm \gamma, \bm s) \leq0\\
    \end{array}\right], 
\end{equation}
\begin{equation}
        \ell_{\bm \theta}^\text{sub}(\bm x,\bm s) := \left[
    \begin{array}{rl}
    \min  & \|(\gamma_{f},\gamma_o)\|\\
    {\rm s.t.}  & \gamma_o, \gamma_f\in\mathbb{R}_+\\
    &g_{\bm \theta} (\bm x,\bm s) \leq  \gamma_{f}\\
    &\displaystyle  J_{\bm\theta}(\bm x, \bm s) \leq \gamma_o
    \end{array}\right].
\end{equation}
\end{subequations}
Figure~\ref{fig:Losses} illustrates the two loss functions. The $\bm\gamma$ variable (red) in the predictability loss $\ell_{\bm\theta}^\text{p}(\bm x,\bm s)$ allows for repositioning the observed $\bm x$ to achieve feasibility and optimality while penalizing the extent of adjustment. Note that $\bm \gamma$ doesn't merely project $\bm x$ into the feasible region $g_{\bm\theta}(\cdot,\bm s)$ but balances between infeasibility and suboptimality. In the suboptimality loss $\ell_{\bm \theta}^\text{sub}(\bm x,\bm s)$, $\gamma_f$ and $\gamma_o$ (light blue) act as slack variables, regulating the levels of infeasibility and suboptimality independently.

\begin{figure}[h!]
 \centering
 \includegraphics[scale=1]{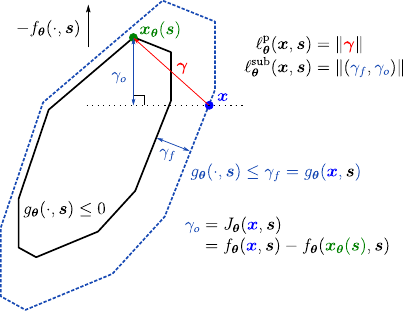}
 \caption{Pictorial representation of the predictability and suboptimality loss functions.}
    \label{fig:Losses}
\end{figure}

The predictability loss was first proposed in \cite{aswani2018inverse} using a slightly different formulation. The naming convention  can be explained by defining variable $\bm y=\bm x+\bm \gamma$, where $\bm \gamma$ shifts $\bm x$, leading to the objective function $\min_{\bm y\in\mathcal{X}(\bm s)} , \|\bm y-\bm x\|$. Here, $\mathcal{X}(\bm s)$ denotes the set of optimal solutions of problem~\eqref{recovered_p}. Consequently, the loss penalizes the discrepancy between the observed decision $\bm x$ and the potential predicted decision $\bm y$. Unlike the loss proposed in \cite{mohajerin2018data} which assumed known constraints, the suboptimality loss extends the loss to unknown constraints penalizing both infeasibility and suboptimality. 

The following proposition shows that the proposed loss functions \eqref{losses} are well defined, in the sense that for any $\bm\theta\in\Theta$ and  $(\bm x,\bm s)$ in the training dataset, the loss $\ell_{\bm \theta}(\bm x,\bm s)$ is zero if and only if $\bm x$ is also an optimal solution of \eqref{recovered_p}. In other words, given a rich enough hypothesis class for $g_{\bm\theta}$ and $f_{\bm\theta}$, the set of optimal solutions of \eqref{recovered_p} coincides with the set of observed optimal solutions of \eqref{forward_problem}. Moreover, the statement implies that if the optimizer $\bm \theta^*$ of \eqref{saa_loss} does not achieve a zero loss, i.e., there exists  $(\bm x, \bm s)$ in the training data such that  $\ell_{\bm \theta^*}(\bm x,\bm s) > 0$, then there is no other $\bm\theta$ in the hypothesis class that perfectly describes the measured pair $(\bm x,\bm s)$. We term this property as \emph{full characterization}. 

\begin{proposition}[Full characterization]\label{prop:full}
For any $\bm \theta\in\Theta$ and the IO data pair~$(\bm x,\bm s)$ generated by the forward model~\eqref{forward_problem}, both predictability~$\ell_{\bm\theta}^\text{p}(\bm x,\bm s)$ and suboptimality~$\ell_{\bm\theta}^\text{sub}(\bm x,\bm s)$ loss defined in~\eqref{losses} satisfy
\begin{equation*}\label{proposition_full_characterization}
     \ell_{\bm \theta}(\bm x,\bm s) =0 \iff \bm x \in \left[
     \begin{array}{cl}
          \displaystyle\argmin_{\bm y\in\mathbb{R}^n}&  f_{\bm \theta}(\bm y,\bm s)  \\
          \text{\em s.t.}&  g_{\bm \theta}(\bm y,\bm s) \leq 0
     \end{array}
     \right].
\end{equation*}
\end{proposition}

\subsection{Measuring Performance}\label{sec:perf}
In this subsection, we define metrics for the goodness of fit in an out-of-sample evaluation. In general, there are two ways to measure the performance. The first method is to directly measure the feasibility and optimality of the optimal solutions generated by learned optimization problems. To this end, for fixed $\bm\theta'\in\Theta$, define $\bm x^*_{\bm\theta'}(\bm s)$ to be an optimizer from problem \eqref{recovered_p}. Thus we evaluate the \emph{true loss} $\sum_{i=1}^{N_\text{out}}\ell(\bm x^*_{\bm\theta'}(\bm s_i),\bm s_i)$ using the true predictability loss
\begin{subequations}\label{true_losses}
\begin{equation}\label{trueloss:predictability}
\ell^\text{p}(\bm x,\bm s) := \left[
    \begin{array}{rl}
     \displaystyle\min  & \|\bm \gamma\|\\
   {\rm s.t.} & \bm \gamma\in\mathbb{R}^n\\
   &g(\bm x+\bm \gamma,\bm s) \leq 0\\
    & J(\bm x+\bm \gamma, \bm s) \leq0\\
    \end{array}\right]
\end{equation}
and true suboptimality loss
\begin{equation} \label{trueloss:suboptimality}
\ell^\text{sub}(\bm x,\bm s) := \left[
    \begin{array}{rl}
    \min  & \|(\gamma_{f},\gamma_o)\|\\
    {\rm s.t.}  & \gamma_o\in\mathbb{R}_+,\, \gamma_f\in\mathbb{R}_+\\
    &g (\bm x,\bm s) \leq  \gamma_{f}\\
    & J(\bm x, \bm s) \leq \gamma_o
    \end{array}\right].
\end{equation}
\end{subequations}

In most practical scenarios, we do not have access to the forward problem \eqref{forward_problem}, preventing a direct evaluation of the true performance of $\bm{\theta}$. Nevertheless, we often have additional data $\{(\bm{x}_i, \bm{s}_i)\}_{i=1}^{N{\text{out}}}$ that was not utilized during the training phase. As an alternative, we can straightforwardly compute $\sum_{i=1}^{N_\text{out}}\ell_{\bm{\theta}'}(\bm{x}_i, \bm{s}_i)$ for a selected $\bm{\theta}' \in \Theta$ for both the predictability and suboptimality loss functions. We illustrate the relationships between these four metrics below in Figures~\ref{fig:Losses2}. From Figure~\ref{fig:Losses2} (bottom), it can be seen that if the recovered optimal solution $\bm x_{\bm\theta} (\bm s)$ coincide with the observed optimal $\x$, all losses become zero. Additionally, the true predictability loss $\ell^\text{p}$ and sample-based loss $\ell^\text{p}_{\bm \theta}$ are equivalent when optimal solutions are unique for both true and recovered problems. It is also worth noting that the true suboptimality loss coincides with the ``Smart Predict, then Optimize" (SPO) loss in \cite{ref:Smart}; see Remark~4.4 in \cite{ref:incenter} for more details. 

\begin{figure}[h!]
 \centering
\includegraphics[width=0.9\linewidth]{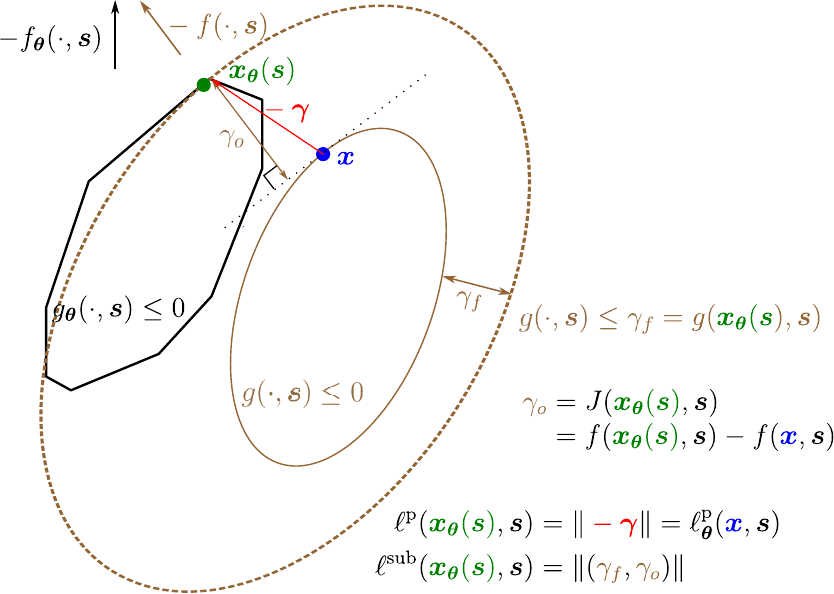}
\includegraphics[width=0.9\linewidth]{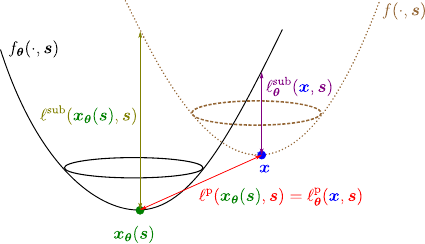}
 \caption{\textbf{Top:}~Pictorial representation of true predictability and suboptimality losses. \textbf{Bottom:}~relationship between estimated out-of-sample loss ($\ell_{\bm\theta}^{\text{p}},\,\ell_{\bm\theta}^{\text{sub}}$) and true out-of-sample loss  ($\ell^{\text{p}},\,\ell^{\text{sub}}$).}
    \label{fig:Losses2}
\end{figure}

\section{Hypothesis Class for $g_{\theta}$ and Reformulations}\label{sec:hypothesis}

In this section, we restrict the admissible constraint function $g_{\bm \theta}$ to a specific hypothesis class and reformulate the predictability and suboptimality losses.  For the remainder of the paper, we assume the objective function is known and focus on the unknown constraints.

Let $\bm A_{\bm \theta}(\bm s)$ and $\bm b_{\bm \theta}(\bm s)$ be a matrix and vector with appropriate dimensions, and restrict the constraint~to
\begin{subequations}\label{hypothesis}
\begin{equation}\label{g_hypothesis}
    g_{\bm \theta}(\bm x,\bm s)= \min_{\bm z\in\mathcal Z}\|\bm x-\bm A_{\bm\theta}( \bm s)\bm z-\bm b_{\bm \theta}(\bm s)\|.
\end{equation}
The function uses a latent variable $\bm z$, which resides in the predetermined conic \emph{primitive set} 
\begin{equation}\label{primitive_set}
    \mathcal{Z} = \{\bm z\in\mathbb{R}^p\,:\, H\bm z -\bm h\in \mathcal{K}\}
\end{equation}
\end{subequations}
where matrices $H\in\mathbb{R}^{l\times p}$ and $\bm h \in\mathbb{R}^l$ and the proper convex cone $\mathcal{K}$ are given.

Intuitively, the hypothesis class \eqref{hypothesis} controls the feasible region of $\bm x$ by manipulating the primitive set $\mathcal{Z}$.  Indeed, we can see that the constraint $g_{\bm \theta}(\bm x,\bm s)\leq 0$ can be reformulated to $\bm x = \bm A_{\bm\theta}( \bm s)\bm z+\bm b_{\bm \theta}(\bm s),\,\bm z\in\mathcal{Z}$. In other words, for each pair $(\bm x,\bm s)$ there exists a $\bm z\in\mathcal{Z}$ that maps to $\bm x$ through the linear map $\bm A_{\bm\theta}( \bm s)\bm z+\bm b_{\bm \theta}(\bm s)$. Hence, through the choice of $\bm\theta$, the matrix $ \bm A_{\bm\theta}(\bm s)$ can scale, rotate and project the primitive set $\mathcal{Z}$, while vector $\bm b_{\bm\theta}(\bm s)$ is responsible for translating the set. An alternative way to view the hypothesis is that the learning model induces the policy $\bm x(\bm s) = \bm A_{\bm\theta}( \bm s)\bm z(\bm s)+\bm b_{\bm \theta}(\bm s)$  for some $\bm z(\bm s)\in\mathcal{Z}$, hence problem~\eqref{recovered_p} aims to learn the policy that best fits the training data. The choice of the primitive set~$\mathcal{Z}$ plays a crucial role in approximation highlighted in the next remark. 

\begin{remark}[Choice of primitive sets]
\label{rem:primitive}
    When the primitive set~$\mathcal{Z}$ is a $p$-dimensional simplex, the resulting policy can be seen as a switched version of $p$ different policies, each represented by a column of~$\bm A_{\bm\theta}( \bm s)$. The latent variable~$\z_i(\bm s)$ acts as the switching mechanism. More broadly, the matrix $\bm A_{\bm\theta}( \bm s)$ projects the high-dimensional simplex (or any other polytopic primitive set) onto the space of $\bm x$. Increasing the dimensions of the primitive set enhances the flexibility of the hypothesis class. Conversely, if the constraints of the forward problem \eqref{forward_problem} are believed to be ellipsoidal, an appropriate choice is $\mathcal{Z} = \{\bm z \in \mathbb{R}^n \mid \|\bm z\|_2 \leq 1\}$, with $\bm A_{\bm\theta}( \bm s)$ rotating and scaling the set accordingly.
\end{remark}

The following example illustrates the richness of the IO models with the constraint class \eqref{hypothesis} by demonstrating $(i)$ how it can learn a forward problem with a discontinuous policy, and $(ii)$ why no IO model with only objective learning (e.g., \cite{mohajerin2018data}) is sufficient to achieve the same.

\begin{example}[Constraint vs. objective learning in IO] 
\label{example1}
    Motivated by power systems, consider 
\begin{equation}\label{example_toy}
    \begin{array}{cl}
        \displaystyle\min_{\bm x\in[0,2]^2}
        & sx_1 + (1-s)x_2 \quad {\rm s.t.} \quad  x_1+x_2=1,
    \end{array}
\end{equation}
where $x_1$ and $x_2$ represent two generators output aiming to meet the demand $x_1 + x_2 = 1$ at the lowest possible cost. The signal $s \in [0,1]$ indicates the per-unit production cost of the first generator, while the cost of the other is proportional to $1-s$. The optimal policy of~\eqref{example_toy} is 
\begin{equation}
\label{example_toy_policy}
(x_1^*(s), x_2^*(s)) = 
\begin{cases} 
(1, 0) & \text{if } s < 0.5, \\
([0, 1], [0, 1]) & \text{if } s = 0.5, \\
(0, 1) & \text{if } s > 0.5,
\end{cases}
\end{equation}
such that $x_1^*(0.5) + x_2^*(0.5) = 1$.
The hypothesis class \eqref{hypothesis} can recover the optimal policy by defining $\mathcal{Z} = \{\bm{z} \in [0,1]^2 : \bm{e}^\top \bm{z} = 1\}$ as the two-dimensional unit simplex, and selecting $A_{\bm{\theta}}(\bm{s}) = I$ (the $2\times 2$ identity) and $\bm{b}_{\bm{\theta}}(\bm{s}) = 0$ such that the resulting policy yields $x_1(s) = z_1(s)$ and $x_2(s) = z_2(s)$. Selecting $z_1(s)$ and $z_2(s)$  as in \eqref{example_toy_policy} satisfies $\mathcal{Z}$ by construction, and recovers the optimal policy. 
However, attempting to learn a quadratic function $f_{\bm{\theta}}(\bm{x}, s) = \bm{x}^\top Q_{\bm{\theta}} \bm{x} + \bm{x}^\top A_{\bm{\theta}} s $ results in the linear policy $\bm{x}(s) = Q_{\bm{\theta}}^{-1} \bm A_{\bm{\theta}} s$ for any $s \in (0,1)$ while saturating at the boundary for any $s = 0$ or $s = 1$, indicating that learning a quadratic cost is indeed insufficient to capture the optimal policy.
\end{example}

\subsection{Reformulation for Linear Objective Functions}
For the remainder of the paper, and without loss of generality, we make the assumption that we can express $A_{\bm \theta}(\bm s)$ and $b_{\bm \theta}(\bm s)$ as affine functions of $\bm s$ through $\bm A_{\bm \theta}(\bm s) = \bm A_0 + \bm A_1s_1 + \ldots + \bm A_Ks_K$ and  $\bm b_{\bm \theta}(\bm s) = \bm b_0 + \bm b_1s_1 + \ldots + \bm b_Ks_K$ where $\bm A_k\in\mathbb{R}^{n\times p}$ and $\bm b_k\in\mathbb{R}^{n}$ for $k=0,\ldots,K$, i.e., $\bm\theta = (\{\bm A_k,\bm b_k\}_{k=0}^K)$. The following theorem provides the reformulation for the learning problem \eqref{saa_loss}. 

\begin{theorem}[Exact reformulation]
\label{thm:convex}
Let $f_{\bm\theta}(\bm x, \bm s) = \bm c(\bm s)^\top \bm x$ and $g_{\bm\theta}$ given in \eqref{hypothesis}. 
The learning problems can be reformulated as follows, using the predictability loss
\begin{subequations}\label{optimization_problems}
    \begin{equation}\label{predictability_problem}
\begin{array}{@{}r@{~}l@{}l}
    \min   & \displaystyle\frac{1}{N} \sum_{i = 1}^N \|\bm \gamma_{i}\|\\
    \textnormal{s.t.}      &\bm A_k\in\mathbb{R}^{n\times p},\, \bm b_k\in\mathbb{R}^{n},\, \forall k\leq K,\\
    &\left.
    \!\!\!\begin{array}{l}
     \bm\gamma_i\in\mathbb{R}^n,\,\bm z_i\in\mathbb{R}^p,\, \bm \lambda_i \in \mathcal K^*\\
    \bm x_i + \bm \gamma_{i} = \bm A_{\bm\theta}(\bm s_i)\bm z_i + \bm b_{\bm\theta}(\bm s_i)\\
     H\bm z_i - \bm h \in \mathcal{K}\\
          \bm c(\bm s_i)^\top\bm A_{\bm\theta}(\bm s_i) - \bm \lambda_i^\top H = 0 \\
     \bm c(\bm s_i)^\top(\bm x_i+\bm\gamma_{i}-\bm b_{\bm\theta}(\bm s_i)) - \bm \lambda_i^\top \bm h \leq 0\\
     \end{array}
    \right\} \forall i\leq N
    \end{array}\\
    \end{equation}
    and using the suboptimality loss
    \begin{equation}\label{suboptimality_problem}
         \begin{array}{@{}r@{~}l@{}l}
    \min   & \displaystyle\frac{1}{N} \sum_{i = 1}^N \|(\gamma_{f,i},\gamma_{o,i})\|\\
    \textnormal{s.t.}      &\bm A_k\in\mathbb{R}^{n\times p},\, \bm b_k\in\mathbb{R}^{n},\, \forall k\leq K,\\
    &\left.
    \!\!\!\begin{array}{l}
    \gamma_{f,i},\gamma_{o,i}\in\mathbb{R}_+, \bm\gamma_i\in\mathbb{R}^n,\,\bm z_i\in\mathbb{R}^p\\
    \bm x_i + \bm \gamma_{i} = \bm A_{\bm\theta}(\bm s_i)\bm z_i + \bm b_{\bm\theta}(\bm s_i)\\
    \|\bm \gamma_i\|\leq \gamma_{f,i}\\
     H\bm z_i - \bm h \in \mathcal{K}, \bm \lambda_i \in \mathcal K^*\\
    \bm c(\bm s_i)^\top\bm A_{\bm\theta}(\bm s_i) - \bm \lambda_i^\top H = 0 \\
     \bm c(\bm s_i)^\top(\bm x_i-\bm b_{\bm\theta}(\bm s_i)) - \bm \lambda_i^\top \bm h \leq \gamma_{o,i}
     \end{array}
    \right\} \forall i\leq N
    \end{array} 
    \end{equation}
\end{subequations}
\end{theorem}
Both problems \eqref{optimization_problems} share similar complexity in the sense that the constraints are linear except for the bilinear term $A_{\bm\theta}(\bm s_i)\bm z_i$. The problem can be efficiently approximated in practice by performing block coordinate descent \cite{MR3444832} on matrices $\{\bm A_k\}_{k=0}^K$ and vectors $\{\bm z_i\}_{i=1}^N$ sequentially until convergence. The gradient descent-based algorithm for the predictability loss in \eqref{optimization_problems} is outlined in Algorithm~\ref{alg:vanilla}. A similar algorithm can be derived for the suboptimality loss. It updates $\{\boldsymbol A_k\}_{k=1}^K$ with gradient descent and solves \eqref{optimization_problems} to compute $\{\bm z_i\}_{i=1}^N$ for fixed $\{\boldsymbol A_k\}_{k=1}^K$. The gradient of the predictability loss \eqref{optimization_problems} with respect to $\{\bm A_k\}_{k=0}^K$ is as follows: For some $(\bm x,\bm s)$, denote $\bm z^*$ as the optimal latent variables, $\bm \beta^*$ as the optimal dual multipliers of constraint $\bm x + \bm \gamma = \bm A_{\bm\theta}(\bm s)\bm z + \bm b_{\bm\theta}(\bm s)$, and $\bm \mu^*$ as the optimal dual multipliers of constraint $\bm c(\bm s)^\top\bm A_{\bm\theta}(\bm s) + \bm \lambda^\top H = 0$. The gradient of the loss function is  
\begin{equation}
    \frac{\partial}{\partial \bm A_k}\ell_{\bm\theta}^\text{p}(\bm x,\bm s) = s_k(\bm c(\bm s)\bm \mu^{*\top} - \bm \beta^*\bm z^{*\top}), \quad \forall k = 0,\ldots,K \,,
\end{equation}
where $s_0 = 1$.  
In practice, the step size  $\eta$ of the gradient descent  can by dynamically updated at each iteration by performing backtracking such as Armijo's rule or other criteria.  

\begin{algorithm}[h!]
\caption{Gradient descent based algorithm for  \eqref{optimization_problems} using predictability loss}\label{alg:vanilla}
\begin{algorithmic}[1]
\STATE \textbf{Initialization: } $\{A^1_k\}_{k = 0}^K$,\, $T$, $\eta$, $t=1$;
\WHILE{$t \leq T$}
\STATE Solve \eqref{optimization_problems} and denote by $\{\bm z_i^*,\bm\beta_i^*,\bm \mu_i^*)\}_{i=1}^N$ the optimal solution and dual multipliers.
\STATE  $\bm A_k^{t+1} = \bm A_k^{t} - \eta\sum_{i=1}^N s_k (\bm c(\bm s_i)\bm \mu_i^\top - \bm \beta_i^*\bm z_i^{*\top}).$
\STATE $t=t+1.$
\ENDWHILE
\end{algorithmic}
\end{algorithm}

\subsection{Adaptive Smoothing}
One can inspect that the optimum of these programs~\eqref{optimization_problems} is typically a nonsmooth function in the variable $\{\bm A_k\}_{k=0}^K$ due to the inner optimization over the multiplicative decision variable~$\{\bm z_i\}_{i=1}^N$. This observation motivates us to deploy smoothing techniques (e.g., Nesterov's smoothing~\cite{nesterov2018lectures}) to improve our algorithm performance. With this in mind, the next remark provides a connection between the two optimization programs~\eqref{optimization_problems} (i.e., predictability vs suboptimality loss), which intuitively sheds light on why the suboptimality loss has often better numerical performance from a computational viewpoint, see numerical experiments in Table~\ref{table:alg_exp} and more details in Appendix~\ref{sec:pro:bi}.

\begin{remark}[Suboptimality loss as smoothed predictability loss] \label{rem:smoothness}
Consider the suboptimality problem~\eqref{suboptimality_problem}. The feasibility variable~$\gamma_{f,i}$ of~\eqref{suboptimality_problem} coincides with Nesterov's smoothing counterpart~\cite{nesterov2005smooth} of the predictability objective function in~\eqref{predictability_problem} when the distance function is consistent with the underlying norm in these programs with an appropriate smoothness parameter. We refer to Appendix~\ref{app:smoothing} for the details to formalize this discussion.   
\end{remark}

Inspired by Remark~\ref{rem:smoothness}, we define smoothed predictability by relaxing the hard constraint involving $\{\bm A_k\}_{k=0}^K$ and penalizing them in the loss function. This is formally defined in the next definition (suboptimality loss follows the same logic and will be presented in Appendix \ref{app:smoothing}). For clarity, we assume $f_{\bm\theta}(\bm x, \bm s) = \bm c(\bm s)^\top \bm x$ and $g_{\bm\theta}$ as given in~\eqref{hypothesis}. 

\begin{definition}[Smoothed predictability loss] 
The $\epsilon$-smoothed counterpart of the predictability loss   \eqref{predictability_problem}   is defined through the optimization program 
\begin{equation}\label{smoothed_formulation}
    \begin{array}{@{}r@{~}l@{}ll}
    \min   & \displaystyle\frac{1}{N} \sum_{i = 1}^n \|\bm \gamma_{i}\| + \epsilon_1 \sum_{i=1}^n \|\bgamma_{s1}\| + \epsilon_2 \sum_{i=1}^n \|\bgamma_{s2}\| \\
    {\rm s.t.}      &\bm A_k\in\mathbb{R}^{n\times p},\, \bm b_k\in\mathbb{R}^{n},\, \forall k=1,\ldots,K,\\
    &\left.
    \!\!\!\begin{array}{l}
     \bm\gamma_i\in\mathbb{R}^n,\,\bm z_i\in\mathbb{R}^p,\, \bm \lambda_i \in \mathcal K^*\\
    \bm x_i + \bm \gamma_{i} = \bm A_{\bm\theta}(\bm s_i)\bm z_i + \bm b_{\bm\theta}(\bm s_i) + \bgamma_{s1}\\
     H\bm z_i - \bm h \in \mathcal{K}\\
     \bm c(\bm s_i)^\top(\bm x_i+\bm\gamma_{i}-\bm b_{\bm\theta}(\bm s_i)) - \bm \lambda_i^\top \bm h \leq 0\\
     \bm c(\bm s_i)^\top\bm A_{\bm\theta}(\bm s_i) - \bm \lambda_i^\top H + \bgamma_{s2} = 0 \\
     \end{array}
    \right\} \, \forall i\leq N.
    \end{array}
\end{equation}
\end{definition}

In order to optimize the smoothed losses~\eqref{smoothed_formulation}, we choose to adaptively increase the penalizing coefficients $\epsilon = (\epsilon_1,\epsilon_2)$. This allows the magnitude of the slack variables $\gamma_{s1}$ and $\gamma_{s2}$ to diminish to zero, thus solving the original predictability loss in \eqref{optimization_problems}. With this in mind, we present the final algorithm in Algorithm~\ref{Alg:smooth}.
\begin{algorithm}[h!]
\caption{Adaptive smoothing algorithm }\label{Alg:smooth}
\begin{algorithmic}[1]
\STATE \textbf{Initialization: } $\{\bm A^1_k\}_{k = 0}^K$,\, $T$, $\eta$, $t=1$, $\epsilon_1$ and $\epsilon_2$ for predictability loss;
\WHILE{$t \leq T$}
\STATE Solve \eqref{smoothed_formulation} and denote by $\{\bm z_i^*,\bm\beta_i^*,\bm \mu_i^*, \bgamma_{s1}^*, \bgamma_{s2}^*\}_{i=1}^N$ the optimal solution, dual multipliers, and smoothing variables.
\STATE  $\bm A_k^{t+1} = \bm A_k^{t} - \eta\sum_{i=1}^N s_k (\bm c(\bm s_i)\bm \mu_i^\top - \bm \beta_i^*\bm z_i^{*\top}).$
\STATE Re-solve problem~\eqref{smoothed_formulation} to get new values of $\{\bgamma^*_{s1},\bgamma^*_{s2}\}_{i=1}^N$
\STATE If the change in $\sum_{i=1}^N \|\bgamma_{s1^*}\|$ or $\sum_{i=1}^N \|\bgamma_{s2^*}\|$ is small enough, increase the value of $\epsilon_1 \text{or } \epsilon_2$.
\STATE $t=t+1.$
\ENDWHILE
\end{algorithmic}
\end{algorithm}

We use simple updating rules for adjusting $(\epsilon_1, \epsilon_2)$ in Algorithm \ref{Alg:smooth}: The initial values are $\epsilon_1 = \epsilon_2=1$, and every time the change in the values of $\sum_{i=1}^N \|\bgamma_{s1}\|^2$ and $\sum_{i=1}^N \|\bgamma_{s2}\|^2$ are less than $0.01/10^{(\log_2(\epsilon_{1})+1)}$, we multiply the parameters~$\epsilon_{1}$ and $\epsilon_{2}$ by~$2$.

The convergence property of Algorithms \ref{alg:vanilla} and \ref{Alg:smooth} is formally summarized in the following proposition.
 \begin{proposition}[Convergence]  \label{prop_convergence} Let $\bm \theta^t = \{\bm A^t_k\}_{k = 0}^K$  where $\{\bm A^t_k\}_{k = 0}^K$ is the outcome of Algorithm 1 (Algorithm~\ref{Alg:smooth}, respectively) after $t$ iterations with the Armijo rule stepsize, and $\{\bm b^t_k\}_{k = 0}^K$  be the solution of the proposed predictability or suboptimality loss function defined in \eqref{optimization_problems} (the smoothed version \eqref{smoothed_formulation1} in Section~\ref{Section_adaptive_smoothing} in the supplementary, respectively) when the matrices $\bm A_k$  are set to the proposed algorithm outcome. Then, the loss function value is monotonically decreasing, i.e., for any pair $(\bm x, \bm s)$  we have $\ell_{\bm \theta^t}(\bm x, \bm s)\geq \ell_{\bm \theta^{t+1}}(\bm x, \bm s)\geq 0$, and hence, it convergences to a finite nonnegative value (local optimal).
 \end{proposition}

We illustrate the difference between Algorithms~\ref{alg:vanilla} and \ref{Alg:smooth} through a noiseless example where the true hypothesis class is covered by the primitive set. The final results are reported in Table~\ref{table:alg_exp}; see also Appendix~\ref{sec:pro:bi} for further related details. It is worth mentioning that when using the vanilla gradient decent Algorithm~\ref{alg:vanilla}, the suboptimality loss has much better performance (since it has better smoothness), whereas leveraging the smooth counterpart of Algorithm~\ref{Alg:smooth} achieves competitive performance with both loss functions (see also Figure~\ref{fig:comp_smooth_reg} for more details concerning the relevant statistics). 
\begin{table}[h!]
    \centering
    \caption{Training loss of Algorithms~\ref{alg:vanilla} and \ref{Alg:smooth} on different losses, where the global optimal is zero.}\label{table:alg_exp}
    \begin{tabular}{|c|c|c|}
    \hline

    loss          & Alg~\ref{alg:vanilla} & Alg~\ref{Alg:smooth} \\
    \hline
 
    Predictability     & $1.90 \pm 0.36$ & $(2.5 \pm 7.2) \times 10^{-4}$ \\

    \hline
    Suboptimality     & $0.03 \pm 0.05$ & $(8.6 \pm 2.1) \times 10^{-4}$ \\
    \hline
    \end{tabular}
\end{table}

\subsection{Convex  and Mixed-integer Reformulations}\label{sec:convexandmilp}

In this section, we discuss choices of $g_{\bm \theta}$ for which problems~\eqref{optimization_problems}  can be cast as convex or mixed-integer programs which can be solved using off-the-shelf solvers. Restricting $\bm A_\theta(\bm s) = \alpha \in\mathbb{R}_+$, i.e., $\bm\theta = (\alpha,\{\bm b_k\}_{k=0}^K)$ will achieve a convex reformulation. This approximation only allows for scaling and translation of the primitive set $\mathcal Z$ and does not permit rotation or projection.  The following proposition provides the convex reformulation of the predictability loss. The suboptimality loss can be reformulated in a similar way.
\begin{proposition}[Tractable convex reformulation]
Let $f_{\bm\theta}(\bm x, \bm s) = \bm c(\bm s)^\top \bm x$ and $g_{\bm\theta}$ defined in~\eqref{hypothesis} with $\bm A_\theta(\bm s) = \alpha\in\mathbb{R}_+$ and  $\bm\theta = (\alpha,\{\bm b_k\}_{k=0}^K)$. Then, the predictability loss \eqref{optimization_problems} can be reformulated as the convex optimization
\begin{equation}\label{convex:predictability} 
    \begin{array}{@{}r@{~}l@{}ll}
    \min   & \displaystyle\frac{1}{N} \sum_{i = 1}^N \|\bm \gamma_{i}\|\\
    \textnormal{s.t.}      &\alpha\in\mathbb{R}_+,\, \bm b_k\in\mathbb{R}^{n},\, \forall k\leq K,\\
    &\left.
    \!\!\!\begin{array}{l}
     \bm\gamma_i\in\mathbb{R}^n,\,\bm \zeta_i\in\mathbb{R}^p,\, \bm \lambda_i \in \mathcal K^*\\
    \bm x_i + \bm \gamma_{i} = \bm \zeta_i + \bm b_{\bm\theta}(\bm s_i)\\
     H\bm \zeta_i - \alpha\bm h \in \mathcal{K}\\
          \bm \alpha c(\bm s_i)^\top - \bm \lambda_i^\top H = 0 \\
     \bm c(\bm s_i)^\top(\bm x_i+\bm\gamma_{i}-\bm b_{\bm\theta}(\bm s_i)) - \bm \lambda_i^\top \bm h \leq 0\\
     \end{array}
    \right\}  \forall i\leq N.\\
    \end{array} 
\end{equation}
\end{proposition}
The auxiliary variable $\bm\zeta$ replaces $\alpha\bm z$ in constraints $\bm x+\bm \gamma = \alpha\bm z +\bm b(\bm s)$ and $H\alpha\bm z - \alpha\bm h\in\mathcal{K}$, where the latter stems from multiplying $H\bm z - \bm h\in\mathcal{K}$ by $\alpha\in\mathbb{R}_+$.   Despite its limitations, this choice of hypothesis class can enhance computational efficiency for large-scale problems.

Restricting either $\bm z$ or $\bm A_{\bm\theta}(\bm s)$ to be binary, allows to reformulate the bilinear term $\bm A_{\bm\theta}(\bm s)\bm z$ using linear inequalities (McCormick inequalities) and formulating the problem as a mixed-interger linear program. The application domain will dictate which of the two will be binary, with two interesting examples emerging. For the first, notice that setting $\mathcal{Z}_{\text{bin}}:=\mathcal Z\cap\{0,1\}^p$ constitutes a restriction to the hypothesis class.   The next observation provides conditions under which the restriction to the integer $\mathcal{Z}$ is done without loss of optimality.
\begin{observation}[Discrete vs. continuous primitive sets]
Let $f_{\bm\theta}(\bm x, \bm s) = \bm c(\bm s)^\top \bm x$ and $g_{\bm\theta}$ given in \eqref{hypothesis} with $\mathcal{Z}_{\text{bin}} = \{\bm z\in\{0,1\}^p\,:\,\bm e^\top \bm z = 1\}$, and let $\bm\theta^*$ be the optimal value of problem~\eqref{optimization_problems} using the predictability loss. If $\bm c(\bm s)^\top\bm A_{\bm\theta^*}(\bm s)$ is not parallel to any of the facets of the simplex for all $\bm s$, then the optimal value of the problem will coincide with the optimal value of the predictability loss problem \eqref{optimization_problems} where $\mathcal{Z} = \{\bm z\in \mathbb R_+^p\,:\,\bm e^\top \bm z = 1\}$.
\end{observation}
For the second case, assume that $g_{\bm\theta}$ given in \eqref{hypothesis} with  $\bm A_{\bm \theta}(\bm s)=\bm{A}\in\{-1,0,1\}^{n\times p}$ and $\bm b_{\bm \theta}(\bm s):\Theta\mapsto\mathbb{R}^{n}$ for all $\bm s$. This choice of the hypothesis class can be useful for learning equality constraints, for example, learning the physical links in a network where $\bm{A}$ dictates the connectivity of the network.  Notice that the number of binary variables needed for the reformulation of the problem is independent to the size of the dataset. We will see an example of this hypothesis class used in Section~\ref{learningnetwork} to learn the structure of a power network.

\section{Numerical Experiments}\label{sec:real}

We apply our methods on an instance of a power system described in \cite{bampou2011scenario}. Both problems have the same forward problem, i.e., a network flow problem with the following formulation. We consider a power system consists of a set of regions $ \mathcal R = \{1, \cdots , 5\}$ with electricity demands $s^{\text{demand}}_r$, $r \in R$. Demands are satisfied by a set $\mathcal N = \{1, 2, 3\}$ of power plans, where each plant $n \in \mathcal N$ produces $x_n$ units of energy at costs $s^{\text{cost}}_n$. Regions are connected by a set $\mathcal M = \{1, \cdots , 5\}$ of directed transmission lines. Each line $m\in \mathcal M$ has a capacity of $\bar{f}_m$ units of energy.  A pictorial representation of the network is given in Figure~\ref{fig:power_system} (top).

The forward problem is formulated as follows where $\mathcal N(r)$ denotes the set of generators in node $r$, with $\mathcal M_+(r)$ and $\mathcal M_-(r)$ denoting the sets of incoming and outgoing flows from node $r$, respectively.
\begin{equation}\label{app:model}
\begin{array}{ll}
    \min &  {\bm s^{\text{cost}}}^\top \bm x \\
    \text{s.t.} & x_n \in \mathbb{R}_+, \, x_n \leq C_n, \, \forall n \in \mathcal{N}, \\
                &  f_m \in \mathbb{R}, \, |f_m| \leq \bar{f}_m, \, \forall m \in \mathcal{M}, \\
                &  \sum_{n \in \mathcal{N}(r)} x_n + \sum_{m \in \mathcal{M}_+(r)} f_m \\
                & \quad = \sum_{m \in \mathcal{M}_-(r)} f_m + s_r^{\text{demand}}, \, \forall r \in \mathcal{R}.
\end{array}
\end{equation}
We set $C_n = 3.5$ for all $n\in\mathcal{N}$ and $\bar{f}_m = 3.5$ for all $m \in \mathcal M$.  We generate $N_{\text{train}} = 100$ data points  for training and $N_{\text{test}} = 200$ data points for testing, by generating signals uniformly at random from $s^{\text{cost}}_1\in[0.2, 1]$, $s^{\text{cost}}_2\in [0.2, 0.5]$,  $s^{\text{cost}}_3\in[1, 2]$ and $s^{\text{demand}}_1\in[0.3, 1.5],\,s^{\text{demand}}_2\in[0.36,1.8],\,s^{\text{demand}}_3\in[0.42,2.1],\,s^{\text{demand}}_4\in [0.48,2.4]$ and $s^{\text{demand}}_5\in [0.54, 2.7]$, and solving problem~\eqref{app:model} to obtain pairs $\{\s_i,\x_i\}_{i = 1}^{N=100}$. Notice that the flow decisions $f_m$ are treated as lurking variables and are not observed.

\subsection{Problem 1: Inferring Generation Policy without Knowing Constraints}\label{sec:power_app1}
In the first experiment, we assume the decision-maker is aware of  the objective function of the forward problem but lacks knowledge of the constraint's structure. This scenario intuitively represents a situation where decision-makers aim to minimize total costs without being aware of any specific business rules. We compare four policies: $(i)$ We use the hypothesis~\eqref{hypothesis} where $\mathcal{Z}$ is the unit simplex of dimension $p\in\{3,6,9\}$ using the Adaptive Smoothing Algorithm~\ref{Alg:smooth} with a limit of 3000 iterations, $(ii)$ we use the convex formulation discussed in Section~\ref{sec:convexandmilp} using the unit simplex of dimension $p = 3$, $(iii)$ we compare against the linear policy induced by learning the quadratic function $f_{\bm{\theta}}(\bm{x}, s) = \bm{x}^\top Q_{\bm{\theta}} \bm{x} + \bm{x}^\top A_{\bm{\theta}} \bm s $, see Example~\ref{example1}, and $(iv)$ we use the hypothesis~\eqref{hypothesis} where $\mathcal{Z}$ is the unit simplex of dimension $p\in\{3,6,9\}$ and solve problems \eqref{optimization_problems} using the non-convex quadratic solver of Gurobi v11.0.3 with a time limit of 1800 seconds. We note that  policy $(iii)$ will coincide with the linear regression policy $\min_{\bm\theta\in\Theta}\sum_{i = 1}^{N_\text{train}} \|\bm x_i - (\bm A_{\bm \theta}(\bm s_i) - \bm b_{\bm\theta}(\bm s_i))\|^2_2$ since the constraints of the problem are assumed unknown. For all methods, we solve both the corresponding predictability and suboptimality loss problems. 
Table~\ref{tab:multicol} shows the out-of-sample performance of each method using the predictability $\ell_\theta^\text{p}$ and suboptimality $\ell_\theta^\text{sub}$ losses, as well as the true predictability $\ell^\text{p}$ and suboptimality $\ell^\text{sub}$ losses. As discussed in Section~\ref{sec:perf}, by definition $\ell_\theta^\text{p} = \ell^\text{p}$. We observe the following: $(i)$ The convex formulation is the most computationally efficient but may yield lower quality solutions compared to other methods. $(ii)$ Learning a quadratic cost is significantly inferior to other methods. $(iii)$ Increasing the dimension $p$ of the unit simplex enhances the hypothesis class flexibility, improving out-of-sample performance for both the predictability and suboptimality loss. $(iv)$ Algorithm 2 performs significantly better compared to using Gurobi in almost all cases, particularly the most complex ones.
We further investigate the performance of the proposed approach using the larger IEEE 14-bus system \cite{leon2020quadratically}. The results are presented in Appendix~\ref{app:ieee14}.

\begin{table}[h!]
\vspace{-0.1in}
\caption{Summary of out-of-sample  performances. }\label{tab:multicol}
\begin{center}
\small{
\begin{tabular}{|c|c|c|c|c|c|}
\multicolumn{2}{c}{Hypothesis Class and Loss}& \multicolumn{3}{c}{Evaluation method}\\
    \hline
    & loss & $\ell_\theta^p /\ell^p$ &  $\ell_\theta^\text{sub}$ &  $\ell^{\text{sub}}$ & time \\
    \hline
    \multirow{2}{*}{Alg~2,\,$p=3$}&  predict. &  $3.20$ & $0.17$ & $0.83$ & $569.0$s\\
       \cline{2-6}
       & subopt. & $5.50$  & $0.18$ & $2.45$ & $430.8$s\\

    \hline
    \multirow{2}{*}{Alg~2,\,$p=6$}& predict. &  $1.91$ & $0.05$ & $0.51$ & $1159.6$s \\
       \cline{2-6}
       & subopt. & $1.50$ & $0.03$ & $0.42$ & $754.1$s \\
    \hline
    \multirow{2}{*}{Alg~2,\,$p=9$}& predict. &  $1.58$ & $0.05$ & $0.69$ & $591.7$s \\
    \cline{2-6}
       &subopt. & $1.61$ & $0.04$ & $0.50$ & $590.6s$\\
    \hline
    \hline
    \multirow{2}{*}{Convex, $p=3$}& predict. &  $3.36$ & $2.07$ & $0.34$ & $0.31$s \\
    \cline{2-6}
       &subopt. & $10.75$  & $0.31$ & $7.72$ & $0.15s$\\
    \hline
        \hline
    \multirow{1}{*}{quadratic cost}& predict. &  $2.58$  & $2.56$ & $0.63$ & $2.54$s \\
    \hline
     \hline
    \multirow{2}{*}{Gurobi $p=3$}&  predict.  &  $6.32$ & $4.44$ & $0.49$ & $1800$s\\
       \cline{2-6}
       & subopt. &  $3.36$ & $0.05$ & $1.30$ & $1800$s\\
    \hline
    \multirow{2}{*}{Gurobi $p=6$}& predict. &  $6.54$ & $4.31$ & $0.76$ & $1800$s\\
       \cline{2-6}
       & subopt. &  $7.76$ & $0.15$ & $4.96$ & $1800$s  \\

    \hline
    \multirow{2}{*}{Gurobi $p=9$}& predict. &  $6.66$ & $4.25$ & $0.77$ & $1800$s \\
    \cline{2-6}
       &subopt. & $11.00$ & $4.63$ & $5.19$ & $1800$s\\
    \hline
\end{tabular}}
\vspace{-0.1in}
\end{center}
\end{table}

\subsection{Problem 2: Learning Power Network Structures}\label{learningnetwork}

In the second experiment, we assume the decision-maker knows the capacity constraints and demand locations but is unaware of the transmission line positions. The goal is to recover the network structure using inverse optimization, in effect learning the last constraint of problem~\eqref{app:model}. In a network with 5 nodes, there are 10 possible transmission line connections, and we aim to identify the configuration that best fits the data. Using the same experimental setup as the previous example, we treat the flow decisions as latent variables. {Leveraging the structure of \eqref{hypothesis}, we can re-express the policy $\bm x(\bm s) = \bm A_{\bm\theta}( \bm s)\bm z(\bm s) + \bm b_{\bm \theta}(\bm s), \bm z(s)\in \mathcal{Z}$ as the following 5 constraints, each corresponding to a node in the network
\begin{equation*}
    \hat{\bm x}(\bm s)   = \bm A  \bm z(\bm s)   + \bm s^{\text{demand}},\,  \bm z(s)\in\mathcal{Z} = [-\bar{f}_m,\bar{f}_m]^{10},
\end{equation*}
where $\hat{\bm x}(\bm s) = [x_1(\bm s),0,x_2(\bm s),0,x_3(\bm s)]^\top$ denotes the injection of energy at in the network.
\begin{figure}[h!]
    \centering
   \includegraphics[width=0.8\linewidth]{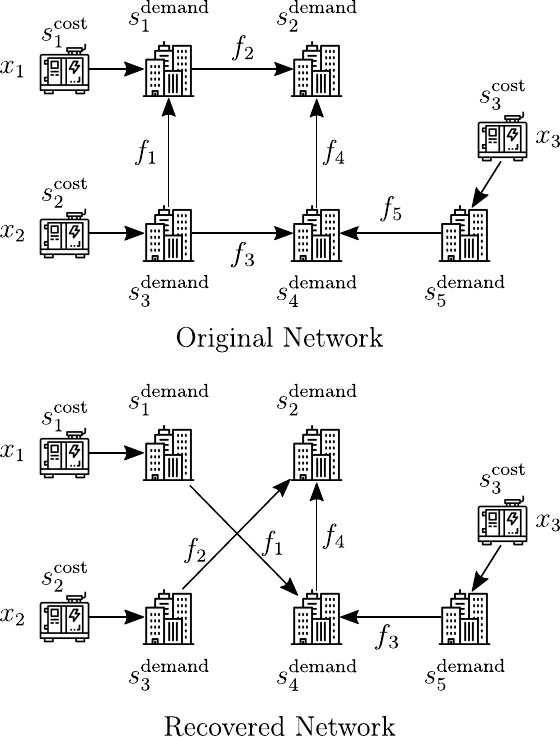}
    \caption{\textbf{Top:} Power network topology used in the forward problem in Section~\ref{sec:real}. \textbf{Bottom:} Recovered network from inverse optimization in Section~\ref{learningnetwork}.}
    \label{fig:power_system}
\end{figure} 
Notice that generators exist in nodes 1, 3 and 5, see Figure~\ref{fig:power_system} (top). The decision matrix $\bm A\in\{0,1\}^{5\times 10}$ controls the connectivity of the network, e.g., if $A_{1,2} = 1$, transmission line 2 emanates from node 1.} {Note that $\bm A$ is subject to additional constraints to ensure that there are only $C_5^2=10$ potential line connections.}
The auxiliary vector $\bm z(\bm s)$ represents the energy flow in the transmission lines. Therefore, we define the primitive set as a hyper-rectangle $\mathcal{Z} = [-\bar{f}_m,\bar{f}_m]^{10}$, which mirrors the transmission capacities. Since we assume that the locations of the generators are known, the given matrix $G$ dictates their position in the system. The inverse optimization problem is presented in Appendix~\ref{app:flow} and optimizes over matrices $\bm A$ 
. As discussed in Section~\ref{sec:convexandmilp}, the problem can be cast as a mixed-integer linear program for both the predictability and suboptimality losses. Both approaches achieved zero loss indicating that the recovered problem is able to exactly describe the training data. The recovered network is presented in Figure~\ref{fig:power_system} (bottom) in the appendix. It is interesting to note that the two networks have different transmission line configuration, with the recovered network having 4 lines instead of 5 that the original network has. Nevertheless, even with 4 lines, the recovered problem is able exactly match the observed data and achieve zero loss. Finally, we compare our approach with the state-of-the-art method for learning problem constraints proposed by \cite{aswani2018inverse}, which also utilizes predictability loss. Their method involves enumerating all potential solutions, assessing the loss for each configuration. It is worth noting that there are 10 potential line connections, each of which can either be absent or present, resulting in $2^{10}$ configurations. We estimated that evaluating each configuration takes on average of $2.64$ seconds, thus enumerating all possibilities would take approximately $2.64 \times 2^{10}$ seconds to find the global optimum. In contrast, our proposed method solves the training problem in just $23.5$ seconds.

\section{Limitations and Future Work}\label{sec:conclusion}
In closing, we acknowledge several limitations: $(i)$ Our method supports only known linear objectives, excluding more complex structures like quadratic objectives which will be interesting to examine further. However, as shown in Example~\ref{example1}, the proposed approach does creates complex policy structures, unlike the linear policy induced from learning a quadratic objective. $(ii)$ It is worth exploring further cases where the hypothesis class leads to convex and MILP formulations (similar to Section~\ref{sec:convexandmilp}), reducing reliance on local search algorithms. $(iii)$ Our hypothesis class induces a non-trivial relationship between the choice of the primitive set and resulting policy behavior. Future work will focus on understanding the impact of the primitive set on the policies it generates, akin to the choice of basis functions in classical linear regression.

\section*{Acknowledgements}
This work was partly supported by the European Research Council (ERC) project TRUST-949796.

\section{Impact Statements}
This paper presents work whose goal is to advance the field of Machine Learning. There are many potential societal consequences of our work, none which we feel must be specifically highlighted here.


\bibliography{bibliography_ICML2025}

\begin{thebibliography}{30}
\providecommand{\natexlab}[1]{#1}
\providecommand{\url}[1]{\texttt{#1}}
\expandafter\ifx\csname urlstyle\endcsname\relax
  \providecommand{\doi}[1]{doi: #1}\else
  \providecommand{\doi}{doi: \begingroup \urlstyle{rm}\Url}\fi

\bibitem[Ajayi et~al.(2022)Ajayi, Lee, and Schaefer]{ajayi2022objective}
Ajayi, T., Lee, T., and Schaefer, A.~J.
\newblock Objective selection for cancer treatment: An inverse optimization approach.
\newblock \emph{Operations Research}, 70\penalty0 (3):\penalty0 1717--1738, 2022.

\bibitem[Akhtar et~al.(2021)Akhtar, Kolarijani, and Esfahani]{akhtar2021learning}
Akhtar, S.~A., Kolarijani, A.~S., and Esfahani, P.~M.
\newblock Learning for control: An inverse optimization approach.
\newblock \emph{IEEE Control Systems Letters}, 6:\penalty0 187--192, 2021.

\bibitem[Aswani et~al.(2018)Aswani, Shen, and Siddiq]{aswani2018inverse}
Aswani, A., Shen, Z.-J., and Siddiq, A.
\newblock Inverse optimization with noisy data.
\newblock \emph{Operations Research}, 66\penalty0 (3):\penalty0 870--892, 2018.

\bibitem[Ayer(2015)]{ayer2015inverse}
Ayer, T.
\newblock Inverse optimization for assessing emerging technologies in breast cancer screening.
\newblock \emph{Annals of Operations Research}, 230:\penalty0 57--85, 2015.

\bibitem[Bampou \& Kuhn(2011)Bampou and Kuhn]{bampou2011scenario}
Bampou, D. and Kuhn, D.
\newblock Scenario-free stochastic programming with polynomial decision rules.
\newblock In \emph{2011 50th IEEE Conference on Decision and Control and European Control Conference}, pp.\  7806--7812, 2011.

\bibitem[B{\"a}rmann et~al.(2017)B{\"a}rmann, Pokutta, and Schneider]{barmann2017emulating}
B{\"a}rmann, A., Pokutta, S., and Schneider, O.
\newblock Emulating the expert: inverse optimization through online learning.
\newblock In \emph{Proceedings of the 34th International Conference on Machine Learning-Volume 70}, pp.\  400--410, 2017.

\bibitem[Bertsekas(2015)]{bertsekas2015convex}
Bertsekas, D.
\newblock \emph{Convex optimization algorithms}.
\newblock Athena Scientific, 2015.

\bibitem[Bertsekas(1999)]{MR3444832}
Bertsekas, D.~P.
\newblock \emph{Nonlinear programming}.
\newblock Athena Scientific Optimization and Computation Series. Athena Scientific, Belmont, MA, second edition, 1999.

\bibitem[Bertsimas et~al.(2015)Bertsimas, Gupta, and Paschalidis]{bertsimas2015data}
Bertsimas, D., Gupta, V., and Paschalidis, I.~C.
\newblock Data-driven estimation in equilibrium using inverse optimization.
\newblock \emph{Mathematical Programming}, 153:\penalty0 595--633, 2015.

\bibitem[Chan \& Kaw(2020)Chan and Kaw]{chan2020inverse}
Chan, T.~C. and Kaw, N.
\newblock Inverse optimization for the recovery of constraint parameters.
\newblock \emph{European Journal of Operational Research}, 282\penalty0 (2):\penalty0 415--427, 2020.

\bibitem[Chen et~al.(2021)Chen, Chen, and Langevin]{chen2021inverse}
Chen, L., Chen, Y., and Langevin, A.
\newblock An inverse optimization approach for a capacitated vehicle routing problem.
\newblock \emph{European Journal of Operational Research}, 295\penalty0 (3):\penalty0 1087--1098, 2021.

\bibitem[Dempe \& Lohse(2006)Dempe and Lohse]{dempe2006inverse}
Dempe, S. and Lohse, S.
\newblock Inverse linear programming.
\newblock In \emph{Recent Advances in Optimization}, pp.\  19--28. Springer, 2006.

\bibitem[Dimanidis et~al.(2025)Dimanidis, Ok, and Mohajerin~Esfahani]{ref:offlineRL}
Dimanidis, I., Ok, T., and Mohajerin~Esfahani, P.
\newblock Offline reinforcement learning via inverse optimization.
\newblock \emph{preprint available at arXiv:2502.20030}, 2025.

\bibitem[Dong et~al.(2018)Dong, Chen, and Zeng]{dong2018generalized}
Dong, C., Chen, Y., and Zeng, B.
\newblock Generalized inverse optimization through online learning.
\newblock \emph{Advances in Neural Information Processing Systems}, 31, 2018.

\bibitem[Elmachtoub \& Grigas(2022)Elmachtoub and Grigas]{ref:Smart}
Elmachtoub, A.~N. and Grigas, P.
\newblock Smart “predict, then optimize”.
\newblock \emph{Management Science}, 68\penalty0 (1):\penalty0 9--26, 2022.

\bibitem[Ghobadi \& Mahmoudzadeh(2020)Ghobadi and Mahmoudzadeh]{ghobadi2020inferring}
Ghobadi, K. and Mahmoudzadeh, H.
\newblock Inferring linear feasible regions using inverse optimization.
\newblock \emph{European Journal of Operational Research}, 2020.

\bibitem[G{\"u}ler \& Hamacher(2010)G{\"u}ler and Hamacher]{guler2010capacity}
G{\"u}ler, c. and Hamacher, H.~W.
\newblock Capacity inverse minimum cost flow problem.
\newblock \emph{Journal of Combinatorial Optimization}, 19\penalty0 (1):\penalty0 43--59, 2010.

\bibitem[Keshavarz et~al.(2011)Keshavarz, Wang, and Boyd]{keshavarz2011imputing}
Keshavarz, A., Wang, Y., and Boyd, S.
\newblock Imputing a convex objective function.
\newblock In \emph{2011 IEEE international symposium on intelligent control}, pp.\  613--619, 2011.

\bibitem[Leon et~al.(2020)Leon, Bretas, and Rivera]{leon2020quadratically}
Leon, L.~M., Bretas, A.~S., and Rivera, S.
\newblock Quadratically constrained quadratic programming formulation of contingency constrained optimal power flow with photovoltaic generation.
\newblock \emph{Energies}, 13\penalty0 (13):\penalty0 3310, 2020.

\bibitem[Li(2021)]{li2021inverse}
Li, J. Y.-M.
\newblock Inverse optimization of convex risk functions.
\newblock \emph{Management Science}, 67\penalty0 (11):\penalty0 7113--7141, 2021.

\bibitem[Lu et~al.(2018)Lu, Wang, Wang, Ai, and Wang]{lu2018data}
Lu, T., Wang, Z., Wang, J., Ai, Q., and Wang, C.
\newblock A data-driven stackelberg market strategy for demand response-enabled distribution systems.
\newblock \emph{IEEE Transactions on Smart Grid}, 10\penalty0 (3):\penalty0 2345--2357, 2018.

\bibitem[Mohajerin~Esfahani et~al.(2018)Mohajerin~Esfahani, Shafieezadeh-Abadeh, Hanasusanto, and Kuhn]{mohajerin2018data}
Mohajerin~Esfahani, P., Shafieezadeh-Abadeh, S., Hanasusanto, G.~A., and Kuhn, D.
\newblock Data-driven inverse optimization with imperfect information.
\newblock \emph{Mathematical Programming}, 167:\penalty0 191--234, 2018.

\bibitem[Nesterov(2005)]{nesterov2005smooth}
Nesterov, Y.
\newblock Smooth minimization of non-smooth functions.
\newblock \emph{Mathematical programming}, 103:\penalty0 127--152, 2005.

\bibitem[Nesterov et~al.(2018)]{nesterov2018lectures}
Nesterov, Y. et~al.
\newblock \emph{Lectures on Convex Optimization}, volume 137.
\newblock Springer, 2018.

\bibitem[Saez-Gallego et~al.(2016)Saez-Gallego, Morales, Zugno, and Madsen]{saez2016data}
Saez-Gallego, J., Morales, J.~M., Zugno, M., and Madsen, H.
\newblock A data-driven bidding model for a cluster of price-responsive consumers of electricity.
\newblock \emph{IEEE Transactions on Power Systems}, 31\penalty0 (6):\penalty0 5001--5011, 2016.

\bibitem[Tan et~al.(2020)Tan, Terekhov, and Delong]{tan2020learning}
Tan, Y., Terekhov, D., and Delong, A.
\newblock Learning linear programs from optimal decisions.
\newblock \emph{arXiv preprint arXiv:2006.08923}, 2020.

\bibitem[Xu et~al.(2016)Xu, Deng, Hu, Song, and Wang]{xu2016data}
Xu, Z., Deng, T., Hu, Z., Song, Y., and Wang, J.
\newblock Data-driven pricing strategy for demand-side resource aggregators.
\newblock \emph{IEEE Transactions on Smart Grid}, 9\penalty0 (1):\penalty0 57--66, 2016.

\bibitem[Zattoni~Scroccaro et~al.(2024)Zattoni~Scroccaro, Atasoy, and Mohajerin~Esfahani]{ref:incenter}
Zattoni~Scroccaro, P., Atasoy, B., and Mohajerin~Esfahani, P.
\newblock Learning in inverse optimization: Incenter cost, augmented suboptimality loss, and algorithms.
\newblock \emph{Operations Research}, 2024.

\bibitem[Zattoni~Scroccaro et~al.(2025)Zattoni~Scroccaro, van Beek, {Mohajerin Esfahani}, and Atasoy]{ref:Amazon}
Zattoni~Scroccaro, P., van Beek, P., {Mohajerin Esfahani}, P., and Atasoy, B.
\newblock Inverse optimization for routing problems.
\newblock \emph{Transportation Science}, 59\penalty0 (2):\penalty0 301--321, 2025.

\bibitem[Zhang \& Paschalidis(2017)Zhang and Paschalidis]{zhang2017data}
Zhang, J. and Paschalidis, I.~C.
\newblock Data-driven estimation of travel latency cost functions via inverse optimization in multi-class transportation networks.
\newblock In \emph{2017 IEEE 56th Annual Conference on Decision and Control (CDC)}, pp.\  6295--6300, 2017.

\end{thebibliography}
\bibliographystyle{icml2025}

\newpage
\appendix
\onecolumn

\section{Proof of Proposition \ref{prop:full}}\label{app:proof:full}
\begin{proof}
We first prove the statement for the predictabilily loss. Assume that there exists  $\bm \theta\in\Theta$ and $(\bm x,\bm s)$ in the training dataset such that $\ell_{\bm\theta}^\text{p}(\bm x,\bm s) = 0$. This implies that there exists  $\bm\gamma = 0$ since the objective of the predictability loss in \eqref{losses} is $\|\bm\gamma\|$. Thus,  the constraints are satisfied with $g_{\bm \theta}(\bm x,\bm s) \leq 0$ which implies that $\bm x$ is feasible in \eqref{recovered_p}. Moreover, by construction  it implies that 
\begin{equation}\label{prop1:eq}
    f_{\bm\theta}(\bm x,\bm s) \geq \left[
    \begin{array}{rl}
    \displaystyle\min_{\bm y\in\mathbb{R}^n}& f_{\bm \theta}(\bm y,\bm s)\\
    {\rm s.t.} & g_{\bm \theta}(\bm y,\bm s) \leq  0\\ 
\end{array}\right].
\end{equation}
In addition, since $\bm\gamma = 0$ it implies that $J_{\bm\theta}(\bm x, \bm s) \leq0$, which combined with \eqref{prop1:eq} imply that indeed $\bm x$ is an optimizer of~ \eqref{recovered_p}. 

To prove the reverse implication, consider any optimal $\bm x$ from \eqref{recovered_p}. Since $\bm x$ is feasible and optimal in \eqref{recovered_p} it implies that  $\bm\gamma=0$ is feasible in \eqref{losses}. Moreover, since the objective function in \eqref{losses} is $\|\bm\gamma\|$, there does not exist another $\bm\gamma\in\mathbb{R}^n$ that achieves a lower objective, thus $\bm\gamma=0$ is also optimal, hence $\ell_{\bm\theta}^\text{p}(\bm x,\bm s)=0$, which concludes the proof. The proof for the suboptimality loss follows similar arguments.
\end{proof}

\section{Proof of Theorem \ref{thm:convex}} \label{app:proof:reformulation}
The proof of Theorem \ref{thm:convex} can be obtained from the following lemma.
\begin{lemma}\label{lemma:reformulation}
Let $f_{\bm\theta}(\bm x, \bm s) = \bm c(\bm s)^\top \bm x$. The constraints of the predictability loss in \eqref{losses} can be reformulated as follows
\begin{subequations}
\begin{equation}\label{reformulation:predictability}
\begin{array}{l}
g_{\bm \theta}(\bm x+\bm \gamma,\bm s)\leq 0 \quad \iff \quad  \left\{
    \begin{array}{l}
    \exists \bm z \in \mathcal Z\\
\bm x+\bm \gamma  = \bm A_{\bm\theta}(\bm s)\bm z+\bm b_{\bm\theta}(\bm s)
    \end{array}\right.\\~\\  
    J_{\bm \theta}(\bm x+\bm \gamma,\bm s)\leq 0  
\quad \iff \quad 
\left\{
    \begin{array}{l}
    \exists \bm \lambda\in\mathcal{K}^*\\
     \bm c(\bm s)^\top(\bm x+\bm\gamma -  \bm b_{\bm\theta}(\bm s)) - \bm \lambda^\top \bm h \leq 0\\
     \bm c(\bm s)^\top\bm A_{\bm\theta}(\bm s) - \bm \lambda^\top H = 0\\
     \end{array}\right.\\ 
\end{array}
\end{equation}
 and the constraints of the suboptimality loss in \eqref{losses} can be reformulated as follows
 \begin{equation}
\begin{array}{l}
g_{\bm \theta}(\bm x,\bm s)\leq  \gamma_f 
\quad \iff \quad 
\left\{
    \begin{array}{l}
    \exists \bm z \in \mathcal Z,\, \bm \gamma\in\mathbb{R}^n\\
\bm x+ \bm \gamma  = \bm A_{\bm\theta}(\bm s)\bm z+\bm b_{\bm\theta}(\bm s)\\ 
\|\bm \gamma\|\leq \gamma_f
    \end{array}\right.\\~\\
    
J_{\bm \theta}(\bm x ,\bm s)\leq \gamma_o 
\quad \iff \quad 
\left\{
    \begin{array}{l}
    \exists \bm \lambda\in\mathcal{K}^*\\
     \bm c(\bm s)^\top(\bm x -  \bm b_{\bm\theta}(\bm s)) - \bm \lambda^\top \bm h \leq \gamma_o\\
     \bm c(\bm s)^\top\bm A_{\bm\theta}(\bm s) - \bm \lambda^\top H = 0\\
     \end{array}\right.\\   
\end{array}
\end{equation}
\end{subequations}
where the norm used is the same as in the definition of $g_{\bm\theta}$.
\end{lemma}
\begin{proof}
Recall that the definition of  $ g_{\bm \theta}(\bm x,\bm s)$ is 
\begin{subequations}
\begin{equation}
    g_{\bm \theta}(\bm x,\bm s)= \min_{\bm z\in\mathcal Z}\|\bm x-\bm A_{\bm\theta}( \bm s)\bm z-\bm b_{\bm \theta}(\bm s)\|.
\end{equation}
The function uses a latent variable $\bm z$, which resides in the predetermined conic \emph{primitive set} 
\begin{equation}
    \mathcal{Z} = \{\bm z\in\mathbb{R}^p\,:\, H\bm z -\bm h\in \mathcal{K}\}
\end{equation}
\end{subequations}
Plug the above definition into $g_{\bm \theta}(\bm x+\bm \gamma,\bm s)\leq 0$ gives:
\begin{align*}
      &\exists \bm z \in \mathcal Z,\, \bm \gamma\in\mathbb{R}^n\\
&\bm x+ \bm \gamma  = \bm A_{\bm\theta}(\bm s)\bm z+\bm b_{\bm\theta}(\bm s).
\end{align*}
The reformulation of $J_{\bm \theta}$ is done via the dualization. 

\begin{align*}
    J_{\bm \theta}(\bm x+\bm \gamma,\bm s)\leq 0  
\end{align*}
is equivalent to
\begin{align*}
     \begin{array}{rl}
&\bm c(\bm s)^\top(\bm x+\bm\gamma) -  \left[
    \begin{array}{cl}
\displaystyle\min  & \bm c(\bm s)^\top\bm y\\
{\rm s.t.}& \bm y\in\mathbb{R}^n,\, \\
&\bm y   = \bm A_{\bm\theta}(\bm s)\bm z+\bm b_{\bm\theta}(\bm s)\\
& H\bm z - \bm h \in \mathcal{K},\, \bm \gamma\in\mathbb{R}^n
    \end{array}
    \right]
    \end{array} \leq 0
\end{align*}

After dualizing the second term, we obtain:
\begin{align}
    &\bm c(\bm s)^\top(\bm x+\bm\gamma) - \bm c(\bm s)^\top \bm b_{\bm\theta} - \bm \lambda^\top\bm h \leq 0 \\
    & \bm c(\bm s)^\top \bm A_{\bm\theta}(\bm s) - \bm \lambda^\top H = 0
\end{align}

The proof of suboptimality follows the same procedure.

\end{proof}

{
\section{Proof of Proposition \ref{prop_convergence}} 
\begin{proof}
The proof for both Algorithms \ref{alg:vanilla} and \ref{Alg:smooth} follows the same arguments. First, note that the monotonicity of the proposed algorithms' outcome is a straightforward consequence of its coordinate descent nature; see \citep[Section 6.5]{bertsekas2015convex} for similar techniques. More specifically, a classical result of gradient decent algorithms with several popular stepsizes, including the Armijo rule used in this work, ensures that the desired loss function is monotonically decreasing over the iterations \citep[Section 2.1]{bertsekas2015convex}. This observation of the gradient descent over the coordinate $\bm A$  of $\bm \theta$ , along with solving the convex optimization in \eqref{optimization_problems} over the coordinate $\bm b$ and the fact that the loss function is uniformly nonnegative (bounded from below), concludes that the loss function remains monotonically non-increasing across the iterations. 
\end{proof}
}

\section{MILP formulation and Proposition \ref{prop:milp}} \label{app:prop:milp}

\begin{proposition}[MILP reformuation]
\label{prop:milp}
Let $f_{\bm\theta}(\bm x, \bm s) = \bm c(\bm s)^\top \bm x$ and $g_{\bm\theta}$ given in \eqref{hypothesis} with $\mathcal{Z} = \{\bm z\in\{0,1\}^p\,:\,\bm e^\top \bm z = 1\}$. The predictability loss in \eqref{losses} can be reformulated as
\begin{equation}\label{reform:bilinear1}
     \begin{array}{rl}
     \displaystyle\min  & \|\bm \gamma\|\\
   \textnormal{s.t.} & \bm \gamma\in\mathbb{R}^n,\, \bm z \in \{0,1\}^p,\,  \lambda\in\mathbb{R}\\
&\bm x+\bm \gamma  = \bm A_{\bm\theta}(\bm s)\bm z+\bm b(\bm s)\\ 
&\bm e^\top \bm z = 1\\
    & \bm c(\bm s)^\top(\bm x+\bm\gamma) - \bm c(\bm s)^\top\bm b(\bm s) +  \lambda\leq 0\\
    & \bm c(\bm s)^\top\bm A_{\bm\theta}(\bm s) +   \lambda  \bm e^\top = 0\\
    \end{array}
\end{equation}
\end{proposition}

\begin{proof}
The predictability loss with $f_{\bm\theta}(\bm x, \bm s) = \bm c(\bm s)^\top \bm x$ and $g_{\bm\theta}$ given in \eqref{hypothesis} with $\mathcal{Z} = \{\bm z\in\mathbb{R}^p_+\,:\,\bm e^\top \bm z = 1\}$ can be written as
    \begin{equation}\label{reform:bilinear_prop}
     \begin{array}{rl}
     \displaystyle\min  & \|\bm \gamma\|\\
   {\rm s.t.} & \bm \gamma\in\mathbb{R}^n,\, \bm z \in \mathbb{R}_+^p\\
&\bm x+\bm \gamma  = \bm A(\bm s)\bm z+\bm b(\bm s), \quad \bm e^\top \bm z = 1\\
&\bm c(\bm s)^\top(\bm x+\bm\gamma) \leq  \left[
    \begin{array}{cl}
\displaystyle\min  & \bm c(\bm s)^\top\bm y\\
{\rm s.t.}& \bm y\in\mathbb{R}^n,\, \bm z\in\mathbb{R}_+^p\\
&\bm y  = \bm A(\bm s)\bm z+\bm b(\bm s)\\
& \bm e^\top \bm z = 1
    \end{array}
    \right]
    \end{array}
\end{equation}
It is clear that the optimal value of \eqref{reform:bilinear1} constitutes an upper bound on the optimal value of \eqref{reform:bilinear_prop} since $\bm z \in \{0,1\}^p$ is a restriction to $\bm z \in \mathbb{R}_+^p$. Let $\bm \gamma^*$ be an optimal solution of problem~\eqref{reform:bilinear_prop}. We next show that $\bm \gamma^*$ is feasible in \eqref{reform:bilinear1} which will conclude the proof. 

From problem~\eqref{reform:bilinear_prop}, since  $\bm x+\bm \gamma^*$ is a feasible solution in $\bm x+\bm \gamma^*  = \bm A(\bm s)\bm z+\bm b(\bm s),\, \bm e^\top \bm z = 1$, it implies that $\bm x+\bm \gamma^*$ is also feasible in 
\begin{equation}\label{prop:MILPauxiliary}
        \begin{array}{rl}
\displaystyle V_{\bm\theta}(\bm s)=\min  & \bm c(\bm s)^\top\bm y\\
{\rm s.t.}& \bm y\in\mathbb{R}^n,\, \bm z\in\mathbb{R}_+^p\\
&\bm y  = \bm A(\bm s)\bm z+\bm b(\bm s)\\
& \bm e^\top \bm z = 1
    \end{array}
\end{equation}
The last constraint in \eqref{reform:bilinear_prop} also implies that $\bm x+\bm \gamma^*$ is optimal in \eqref{prop:MILPauxiliary}. Additionally, notice that since \eqref{prop:MILPauxiliary} is a linear program and $\bm c(\bm s)^\top\bm A(\bm s)$ is not parallel to any of the facets of the simplex, there exists a unique corner point in the simplex $\bm e^\top \bm z = 1$, i.e.,  $\bm z' \in \{0,1\}^p$, such that $\bm y = \bm A(\bm s)\bm z'+\bm b(\bm s)$   achieves the optimal value. Hence, constraint  $\bm c(\bm s)^\top(\bm x+\bm\gamma) \leq V_{\bm\theta}(\bm s)$ ensures that only $\bm z'$ is feasible, thus   $\bm x+\bm \gamma^*$ is feasible in problem~\eqref{reform:bilinear_prop}, which concludes the proof.
\end{proof}

\section{Smoothing} \label{app:smoothing}

Despite that both the predictability and suboptimality loss are non-convex for the hypothesis class $g_{\bm\theta}$ given in \eqref{hypothesis} with and arbitrary choice of $\mathcal{Z}$, Algorithm 1 tends to be trapped in local minima much more often for the predictability than the suboptimality loss. In this section, we explain this fact by showing that suboptimality loss can be viewed as a smoothed version of predictability loss in the sense of the so-called Nesterov's smoothing. Inspired by this technique, we propose an adaptive smoothing algorithm to potentially escape the local optimum, hence improving the performance.

\begin{definition}[Nesterov’s smoothing \cite{nesterov2018lectures}]\label{def:Nesterov smoothing}
Consider the function in the form of $J(\x) = \max_{\y \in \mathcal{Y}} \langle \bm A\x+\b, \y \rangle - \phi(\y)$ where $\phi$ is a convex function. We define the smooth counterpart of $f_\epsilon$ as
$$J_\epsilon(\x) := \max_{\y \in \mathcal{Y}} \langle \bm A\x+\b, \y \rangle - \phi(\y) - \epsilon d(\y),$$
where $\epsilon >0$ is the smoothing parameter, and $d(\y)$ is called a prox function that (i) is continuous and 1-strongly convex on $\mathcal{Y}$ and (ii) $\min_{\y \in \mathcal{Y}}d(\y)=0$. Then, $J_\epsilon(x)$ is $1/\epsilon$-smooth (i.e., its gradient $\nabla J_\epsilon(\x)$ is $1/\epsilon$-Lipschitz continuous), and  $J_\epsilon(\x) \le J(\x) \le J_\epsilon(\x)+\epsilon/2$ for all $\x \in \mathcal{Y}$. 
\end{definition}

{\bf Suboptimality loss as smoothed predictability loss.}

Considering the optimization problems~\eqref{optimization_problems}, we define~$\bm \beta_i$ as the dual multiplier of the equality constraints~$\bm x_i + \bm \gamma_{i} = \bm A_{\bm\theta}(\bm s_i)\bm z_i + \bm b_{\bm\theta}(\bm s_i)$, and the respective penalization Lagrangian function 
\begin{align}\label{J}
    J(\bm \theta, \bm Z) := \max_{\bm \beta_i} \sum_{i=1}^n \bm \beta_i^\top \big(\bm x_i + \bm \gamma_{i} - \bm A_{\bm\theta}(\bm s_i)\bm z_i- \bm b_{\bm\theta}(\bm s_i)\big) \quad \text{where} \quad \bm Z = \{\gamma_i,z_i\}_{i\le N}\,.
\end{align}
Dualizing this linear constraint in the predictability loss program~\eqref{predictability_problem} reformulates the program to
\begin{align}\label{predictability_problem_J}
\begin{array}{@{}r@{~}l@{}l}
    \min   & \displaystyle\frac{1}{N} \sum_{i = 1}^N \|\bm \gamma_{i}\| + J(\bm \theta, \bm Z)\\
    \textnormal{s.t.}      &\bm A_k\in\mathbb{R}^{n\times p},\, \bm b_k\in\mathbb{R}^{n},\, \forall k\leq K,\\
    &\left.
    \!\!\!\begin{array}{l}
     \bm\gamma_i\in\mathbb{R}^n,\,\bm z_i\in\mathbb{R}^p,\, \bm \lambda_i \in \mathcal K^*\\
     H\bm z_i - \bm h \in \mathcal{K}\\
          \bm c(\bm s_i)^\top\bm A_{\bm\theta}(\bm s_i) - \bm \lambda_i^\top H = 0 \\
     \bm c(\bm s_i)^\top(\bm x_i+\bm\gamma_{i}-\bm b_{\bm\theta}(\bm s_i)) - \bm \lambda_i^\top \bm h \leq 0\\
     \end{array}
    \right\} \forall i\leq N
    \end{array}
\end{align}
Considering the prox function $d(\bm\beta) = {1\over 2}\|\bm \beta + \bm \gamma_i/\epsilon\|^2$, the smoothed version of the function~$J$ is
    \begin{align}\label{J_eps}
        J_{\epsilon}(\bm\theta,\bm Z) =  \sum_{i=1}^N \frac{1}{2\epsilon} \|\bm x_i - \bm A_{\bm\theta}(\bm s_i)\bm z_i- \bm b_{\bm\theta}(\bm s_i)\|^2 - \frac{1}{2\epsilon}\|\bm \gamma_i^2\|. 
    \end{align}
Replacing the above smoothed term in the predictability loss~\eqref{predictability_problem_J} yields the objective function
\begin{align}
    \label{pred_smooth}
    \sum_{i = 1}^N \frac{1}{2\epsilon} \|\bm x_i - \bm A_{\bm\theta}(\bm s_i)\bm z_i- \bm b_{\bm\theta}(\bm s_i)\|^2 + \left(\frac{1}{N} - \frac{1}{2\epsilon}\right)\|\bm \gamma_{i}\|. 
\end{align}
If the smoothing level is $\epsilon = N/2$, and the norm in the objective~\eqref{suboptimality_problem} is separable (i.e., $\|(\gamma_{f,i},\gamma_{o,i})\| = |\gamma_{f,i}|+|\gamma_{o,i}|$), one can then see that the second term in \eqref{pred_smooth} is cancelled and the first term coincides with the optimal solution $\gamma_{f,i}$ in the objective of the suboptimality loss in~\eqref{suboptimality_problem}. In other words, the variable penalizing the feasibility of each data point effectively is a smoothed version of the corresponding term in the predictability loss.

\subsection{Adaptive smoothing}\label{Section_adaptive_smoothing}
We presented the complete version with suboptimality loss in the following.
\begin{definition}
Inspired by the results above , we define smoothed predictability and suboptimality loss in the following. Again, for clarity, we assume $f_{\bm\theta}(\bm x, \bm s) = \bm c(\bm s)^\top \bm x$ and $g_{\bm\theta}$ given in \eqref{hypothesis}. 
The learning problem using the predictability loss can be reformulated as follows
\begin{subequations}\label{smoothed_formulation1}
\begin{equation}
    \begin{array}{rll}
    \min   & \displaystyle\frac{1}{N} \sum_{i = 1}^n \|\bm \gamma_{i}\| + \epsilon_1 \sum_{i=1}^n \|\bgamma_{s1}\| + \epsilon_2 \sum_{i=1}^n \|\bgamma_{s2}\| \\
    {\rm s.t.}      &\bm A_k\in\mathbb{R}^{n\times p},\, \bm b_k\in\mathbb{R}^{n},\, \forall k=1,\ldots,K,\\
    &\left.
    \!\!\!\begin{array}{l}
     \bm\gamma_i\in\mathbb{R}^n,\,\bm z_i\in\mathbb{R}^p,\, \bm \lambda_i \in \mathcal K^*\\
    \bm x_i + \bm \gamma_{i} = \bm A_{\bm\theta}(\bm s_i)\bm z_i + \bm b_{\bm\theta}(\bm s_i) + \bgamma_{s1}\\
     H\bm z_i - \bm h \in \mathcal{K}\\
     \bm c(\bm s_i)^\top(\bm x_i+\bm\gamma_{i}-\bm b_{\bm\theta}(\bm s_i)) - \bm \lambda_i^\top \bm h \leq 0\\
     \bm c(\bm s_i)^\top\bm A_{\bm\theta}(\bm s_i) - \bm \lambda_i^\top H + \bgamma_{s2} = 0 \\
     \end{array}
    \right\} \forall i\leq N\\
    \end{array}
\end{equation}
while the learning problem using the suboptimality loss can be reformulated as follows
\begin{equation}
    \begin{array}{rll}
    \min   & \displaystyle\frac{1}{N} \sum_{i = 1}^n \|(\gamma_{f,i},\gamma_{o,i})\| + \epsilon_1 \sum_{i=1}^n \|\bgamma_{s1}\| \\
    {\rm s.t.}      &\bm A_k\in\mathbb{R}^{n\times p},\, \bm b_k\in\mathbb{R}^{n},\, \forall k=1,\ldots,K,\\
    &\left.
    \!\!\!\begin{array}{l}
    \gamma_{f,i},\gamma_{o,i}\in\mathbb{R}_+\\ \bm\gamma_i\in\mathbb{R}^n,\,\bm z_i\in\mathbb{R}^p,\, \bm \lambda_i \in \mathcal K^*\\
    \bm x_i + \bm \gamma_{i} = \bm A_{\bm\theta}(\bm s_i)\bm z_i + \bm b_{\bm\theta}(\bm s_i) \\
    \|\bm \gamma_i\|\leq \gamma_{f,i}\\
     H\bm z_i - \bm h \in \mathcal{K}\\
     \bm c(\bm s_i)^\top(\bm x_i-\bm b_{\bm\theta}(\bm s_i)) - \bm \lambda_i^\top \bm h \leq \gamma_{o,i}\\
     \bm c(\bm s_i)^\top\bm A_{\bm\theta}(\bm s_i) - \bm \lambda_i^\top H + \bgamma_{s1} = 0 \\
     \end{array}
    \right\} \forall i\leq N\\
    \end{array}
\end{equation}
\end{subequations}
\end{definition}

\subsection{Convergence behavior of adaptive smoothing} \label{app:smooth:exp}

In this section, we study the convergence behavior of Algorithm \ref{Alg:smooth} using the experiment setting described in Section \ref{sec:pro:bi}. We plot the values of training loss and the four metrics defined in \ref{sec:perf} to validate that the proposed method improves the performance of the solution. The plots are shown in Figure \ref{fig:train:smooth}, where the training loss increases as more regularization is added to the optimization. However, the true losses keep decreasing implying better optimal solutions are being found each time.

Finally, it is worth noting that utilizing the algorithm without smoothing leads to inconsistent outcomes, particularly concerning predictability losses. Conversely, employing an adaptive smoothing algorithm here ensures consistent results are obtained. We record this behavior via an example in Figure \ref{fig:comp_smooth_reg}. The experiment settings follow the descriptions introduced in Appendix \ref{sec:pro:bi}, where $p = 5$.

\begin{figure}[htbp] 
\centering
\begin{minipage}[t]{0.44\textwidth} 
  \centering
  \includegraphics[width=1.1\linewidth]{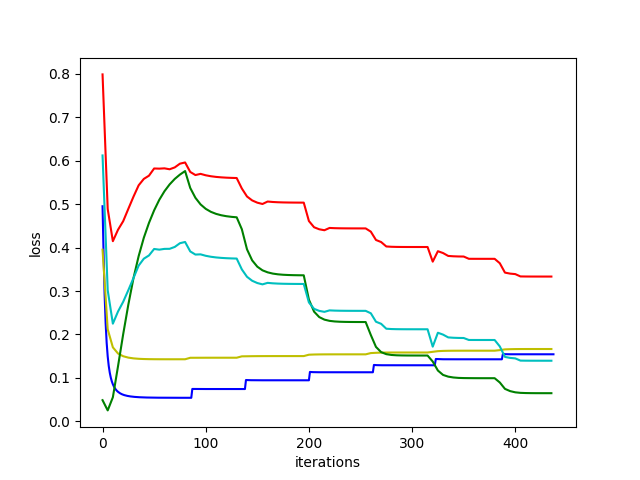}
  \caption{Training loss and test metrics. Blue: Training loss (predictability). Red: Out-of-sample predictability loss. Cyan: True predictability loss. Green: True suboptimality loss. Yellow: Out-of-sample suboptimality loss.}
  \label{fig:train:smooth}
\end{minipage}%
\hspace{0.04\textwidth} 
\begin{minipage}[t]{0.44\textwidth} 
  \centering
  \includegraphics[width=1.1\linewidth]{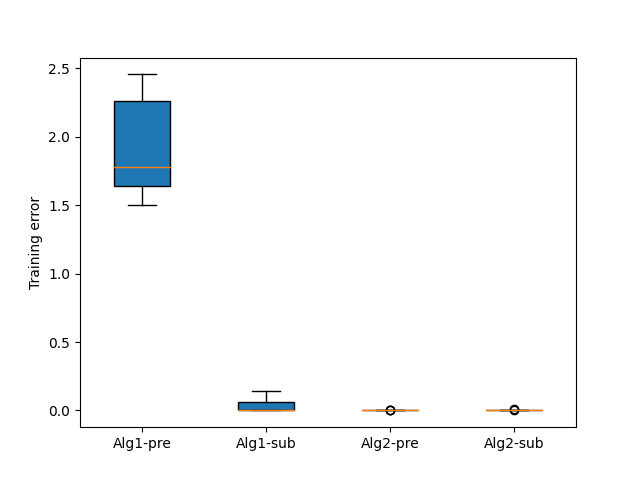}
  \caption{Comparison between convergence behaviors of predictability and suboptimality losses via both Algorithm \ref{alg:vanilla} and \ref{Alg:smooth}. The y-axis represents the achieved training loss. Each box plot is based on the same $10$ randomly generated datasets.}
  \label{fig:comp_smooth_reg}
\end{minipage}
\end{figure}

\section{Problem 1: IEEE 14-Bus System}\label{app:ieee14}

We also apply our methods on IEEE 14-bus system. This power system consists of a set of regions $ \mathcal R = \{1, \cdots , 14\}$ with electricity demands $s^{\text{demand}}_r$, $r \in R$. Demands are satisfied by a set three power plans attached to node 2, node 8, and node 13 respectively. Each plant $n \in \mathcal N$ produces $x_n$ units of energy at costs $s^{\text{cost}}_n$.

\begin{figure}[h!]
    \centering
    \includegraphics[width = 0.6\linewidth]{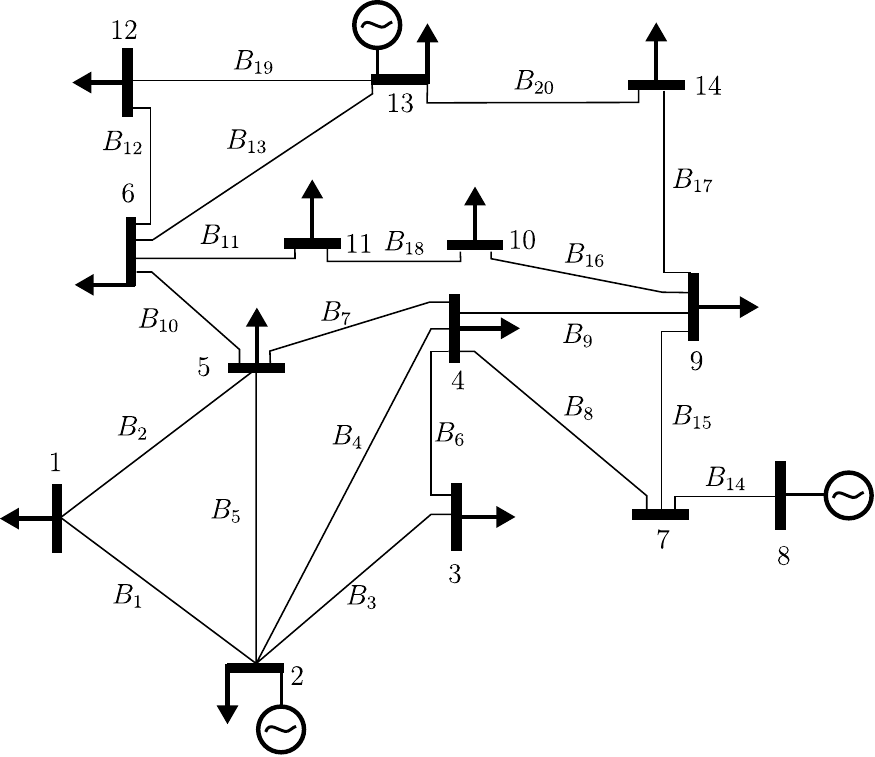}
    \caption{IEEE 14-bus system \cite{leon2020quadratically}. Rectangles denote nodes, arrows denote loads, and circles denotes  generators. }
    \label{fig:ieee-14}
\end{figure}

 We set $C_n = 3.6$ for all generators and $\bar{f}_m = 3$ for all transmission lines.  We generate $N_{\text{train}} = 100$ data points for training and $N_{\text{test}} = 200$ data points for testing, by generating signals uniformly at random from $s^{\text{cost}}_{13}\in[0.2, 1]$, $s^{\text{cost}}_2\in [0.2, 0.5]$,  $s^{\text{cost}}_8\in[1, 2]$, and $s^{\text{demand}}_1\in[0.14, 0.7],\,s^{\text{demand}}_2\in[0.14,0.7],\,s^{\text{demand}}_3\in[0.16, 0.8],\,s^{\text{demand}}_4\in [0.16, 0.8]$, $s^{\text{demand}}_5\in [0.14, 0.7]$, $s^{\text{demand}}_6\in [0.1, 0.5]$, $s^{\text{demand}}_7\in [0.16, 0.8]$, $s^{\text{demand}}_8\in [0.54, 2.7]$, $s^{\text{demand}}_9\in [0.1, 0.2]$, $s^{\text{demand}}_{10}\in [0.12, 0.6]$, $s^{\text{demand}}_{11}\in [0.12, 0.6]$, $s^{\text{demand}}_{12}\in [0.1, 0.5]$, $s^{\text{demand}}_{13}\in [0.1, 0.5]$, and $s^{\text{demand}}_{14}\in [0.12, 0.6]$, solving problem~\eqref{app:model} to obtain pairs $\{\s_i,\x_i\}_{i = 1}^{N=100}$. 

We apply the same method introduced in Section \ref{sec:power_app1} and summarize the results in Table \ref{tab:ieee14}.
\begin{table}[h!]
\caption{Summary of out-of-sample performances on IEEE 14-bus system}\label{tab:ieee14}
\begin{center}
\small{
\begin{tabular}{|c|c|c|c|c|c|}
\multicolumn{2}{c}{Hypothesis Class and Loss}& \multicolumn{3}{c}{Evaluation method}\\
    \hline
    & loss & $\ell_\theta^p /\ell^p$ &  $\ell_\theta^\text{sub}$ &  $\ell^\text{sub}$ & time \\
    \hline
    \multirow{1}{*}{Algo~\ref{Alg:smooth},\,$p=1$}&  predictability  &  $0.83$ & $0.68$ & $0.12$ & $3.5$s\\

    \hline
    \multirow{1}{*}{Algo~\ref{Alg:smooth},\,$p=2$}& predictability &  $0.43$ & $0.06$ & $0.20$ & $148.8$s \\

    \hline
    \multirow{1}{*}{Algo~\ref{Alg:smooth},\,$p=3$}& predictability &  $0.51$ & $0.06$ & $0.24$ & $157.8$s \\
    
        \hline
    \multirow{1}{*}{quadratic cost}& predictability &  $2.58$  & $2.56$ & $0.63$ & $2.54$s \\
    \hline
\end{tabular}}
\end{center}
\vskip -0.1 in
\end{table}

\section{Reformulation of Problem 2} \label{app:flow}
Recall that the primitive set is chosen as a box, i.e., $\mathcal{Z} = \{\bm{z} \in \R^{M} |  -\bar{f}_m \leq z_m \leq \bar{f}_m, \forall m \in M\}$, where $M = \{1,\cdots,|R|=5\}$ ($|R|$ is the total number of the demand locations). For the hypothesis class, we use $\bm A_{\bm \theta} \in \{0,1\}^{5 \times 10}$ and $\bm b_{\bm \theta} = \bm \delta$. Matrix $A_{\bm \theta}$ are  binary matrices indicating the connections of power networks ($C_5^2 = 10$ power lines at most).
We use $\bm G \bm x -\bm g \in \mathcal{G}$ to denote the capacity constraints of the power plants. For clarity, we denote the power generation plan and costs by two $5$-dimensional vectors $\x$ and $\c$. Each dimension corresponds to one node in the graph. For nodes that have no power plant, we simply set its value to zero. Then, the forward problem can be written as:
\begin{equation}
    \begin{aligned}
        \min_{\x} \quad &\c^\top\x \\
        s.t. \quad & \x = \bm A_{\bm \theta} \z + \s^{\text{demand}},\\
        & \bm H \z - \bm h \in \mathcal{K}, \quad (\textbf{equivalently } \mathcal{Z} = \{\bm{z} \in \R^{M} |  -\bar{f}_m \leq z_m \leq \bar{f}_m, \forall m \in M\})\\
        & \bm G \x - \bm g \in \mathcal{G}, \quad (\textbf{equivalently } 0 \leq \x \leq 3.5).
    \end{aligned}
\end{equation}

Then, the predictability loss can be formulated as a binary problem as shown in (\ref{refor:flow}).
\begin{equation}\label{refor:flow}
    \begin{aligned}
    \min_{\bgamma_{i}, \bm z_i, \bm A_{\bm \theta}, \bm \beta_i, \blambda_i} \quad & \frac{1}{n} \sum_{i=1}^N \|\bgamma_{i}\|_2^2\\
    s.t. \quad & 
    \bm x_i + \bgamma_{i} = \bm A_{\bm \theta} \z_i + \s^{\text{demand}}_i , \forall i,\\
    & \bm H \z_i - \bm h \in \mathcal{K}, \forall i,\\
    & \bm G (\x_i+\bgamma_i) - \bm g \in \mathcal{G}, \forall i,\\
    & \c^\top_i(\bm s_i)(\x_i+\bgamma_{i}) - \c^\top(\s_i)\s_i^{\text{demand}} - \blambda_i^\top \bm h +\bm \beta_i^\top \bm g - \bm \beta_i^\top\bm G \bm \s_i^{\text{demand}} \leq 0,\forall i,\\
    & \blambda_i \in \K^*, \bm \beta_i \in \mathcal{G}^*, \forall i,\\
    & \c^\top_i \bm A_{\bm \theta} - \blambda_i^\top \bm H - \bm \beta_i^\top \bm G \bm A_{\bm \theta} = 0, \forall i,\\
    & \bm A_{\bm \theta}  \in \{0,1\}^{5\times {10}}.
    \end{aligned}
\end{equation}

\section{Further Numerical Experiments}\label{sec:comp}
In this section, we conduct experiments on synthetic optimization problems to validate the proposed reformulations under both noiseless and noisy scenarios. We consider the following forward problem.
\begin{equation}\label{synthe:model1}
    \begin{aligned}
    \min_{\bm x} \quad & \bm c^\top \bm x, \quad {\rm s.t.} \quad \|\bm x - \bm e\|_1 \leq h.
    \end{aligned}
\end{equation}
We assume the objective coefficients $\bm c$ are input signals, which follow uniform distributions. The goal is to recover the feasible region based on the observed values of $\bm c$ and $\x$. {\color{black}Notice that choices of distance metrics of the objective functions for both predictability and suboptimality loss do not affect any of the reformulations or algorithms proposed. Throughout the subsequent experiments, square loss will be employed for the sake of simplicity.} 

\subsection{Convex Reformulations}

In this part, we assume $c_i \sim U[-1,1]$, $1 \leq i \leq n = 2$, and we set tje primitive set: $\displaystyle \Z = \{\z \in \mathbb{R}^2 \big| \|\z\|_1 \leq 1\}$ in the hypothesis class.
For the training data, we randomly generate $N$ cost coefficients and calculate the corresponding optimal solutions of the forward problem. Experiments under noise or no noise are both conducted. If the training set contains noise, we add a Gaussian noise $N(0,0.2)$ to each of its dimensions. The performance is evaluated based on a test set that also contains $500$ randomly generated coefficients and solutions. 

When there is no noise, we achieved zero loss for all metrics. The performance summary for the noisy case is summarized in Table~\ref{tab:multicol_convex}. It can be seen that the performance for predictability loss improves with the increasing number of training data under i.i.d. noises. This phenomenon arises due to the statistically consistent nature of estimates generated by predictability loss, as also highlighted in \cite{aswani2018inverse}.

\begin{table}[h!]
\caption{Out-of-sample Performance of Convex formulation with Gaussian noise.}\label{tab:multicol_convex}
\begin{center}
\small{
\begin{tabular}{|c|c|c|c|c|c|}
    \hline
    & loss &$\ell_\theta^p$ & $\ell^p$ &  $\ell_\theta^\text{sub}$ &  $\ell^\text{sub}$ \\
    \hline
    \multirow{2}{*}{$N=100$}& predictability &  $0.54$ & $0.03$ & $0.29$ & $0.05$ \\
       \cline{2-6}
       & suboptimality & $0.66$ & $0.14$ & $0.35$ & $0.28$\\
    \hline
    \hline
    \multirow{2}{*}{$N=500$}& predictability &  $0.52$ & $1.7e^{-3}$ & $0.31$ & $3.3e^{-3}$ \\
       \cline{2-6}
       & suboptimality & $0.61$ & $0.09$ & $0.34$ & $0.17$\\
    \hline
    \hline
    \multirow{2}{*}{$N=1000$}& predictability &  $0.51$ & $0$ & $0.32$ & $0$ \\
    \cline{2-6}
       & suboptimality & $0.59$ & $0.07$ & $0.34$ & $0.13$\\
    \hline
\end{tabular}}
\end{center}
\vskip -0.1in
\end{table}

\subsection{Bilinear Reformulations with Simplex Primitive Set}\label{sec:pro:bi}

In this section, we showcase the performance of our bilinear reformulation, where the primitive set is defined in a high-dimensional space. More specifically, we assume , i.e, $c_i \sim U[0,1]$, $1 \leq i \leq n = 5$, and we consider the following simplex $\Z = \big\{\z \geqslant \bm 0 ~\big|~ \sum_{i=1}^p z_i = 1 \big\}$ in a $p$-dimensional space as the primitive set. For this  bilinear reformulation, we can both apply Algorithm \ref{Alg:smooth} or MILP reformulation to find the optimal solutions.

\textbf{Experiment setting:} We randomly generate $N$ cost coefficients and calculate the corresponding optimal solutions. We consider both predictability loss and suboptimality loss. We use backtracking to adjust the step length according to the Armijo rule.

We first apply Algorithm \ref{Alg:smooth} to optimize the predictability and suboptimality loss. The optimization is stopped after $500$ iterations.
We record the performance of the proposed two losses under noiseless and noisy scenarios. For noisy cases, a Gaussian noise $N(0,0.2)$ is added to each of the dimensions of the original solutions. We also vary the number of observed samples under noisy cases to check the changes in performances. In noiseless cases, we also apply MILP reformulation to obtain the optimal solutions.

\begin{table}[h!]
\caption{Noiseless}\label{tab:multicol:bilinear:noiseless}
\begin{center}
\small{
\begin{tabular}{|c|c|c|c|c|c|c|c|}
    \hline
   n=5 & N=100 & train &$\ell_\theta^p$ & $\ell^p$ &  $\ell_\theta^\text{sub}$ &  $\ell^\text{sub}$& time \\
    \hline
    \multirow{3}{*}{$p=4$}& Predictability & 0.32 & $1.04$ & $1.04$ & $0.44$ & $0.10$ & 174.6s\\
       \cline{2-8}
       & Suboptimality & 0.33 & $1.04$ & $1.04$ & $0.44$ & $0.10$ & $48.3$s\\
       \cline{2-8}
       & MILP & 0.25 & $0.99$ & $0.99$ & $0.45$ & $0.07$ & $901.2$s\\
    \hline
    \hline
    \multirow{3}{*}{$p=5$}& Predictability & 0 &  $0$ & $0$ & $0$ & $0$ & $49$s\\
       \cline{2-8}
       & Suboptimality & 0 & $0$ & $0$ & $0$ & $0$ & $52.3$s\\
       \cline{2-8}
       & MILP & 0 &  $0$ & $0$ & $0$ & $0$ & $902.7$s\\
    \hline
    \hline
    \multirow{3}{*}{$p=6$}& Predictability & 0&  $0$ & $0$ & $0$ & $0$ & $21.6$s\\
    \cline{2-8}
       &Suboptimality & 0& $0$ & $0$ & $0$ & $0$ & $42.3$s\\
       \cline{2-8}
       &MILP & 0&  $0$ & $0$ & $0$ & $0$ & $543.5$s\\
    \hline
\end{tabular}}
\end{center}
\vskip -0.1in
\end{table}

\begin{table}[h!]
\caption{Noisy case with $100$ training samples.}\label{tab:multicol:bilinear:noise100}
\small{
\begin{center}
\begin{tabular}{|c|c|c|c|c|c|c|c|}
    \hline
   $n=5$ & $N=100$ & train &$\ell_\theta^p$ & $\ell^p$ &  $\ell_\theta^\text{sub}$ &  $\ell^\text{sub}$ & time \\
    \hline
    \multirow{3}{*}{$p=4$}& Predictability & $0.52$ & $1.17$ & $0.97$ & $0.60$ & $0.10$ & $87.4$s\\
       \cline{2-8}
       & Suboptimality & $0.43$ & $1.42$ & $1.22$ & $0.57$ & $0.91$ & $85.0$s\\
    \hline
    \hline
    \multirow{3}{*}{$p=5$}& Predictability & $0.14$ &  $0.42$ & $0.24$ & $0.18$ & $0.11$ & $110.9$s\\
       \cline{2-8}
       & Suboptimality & $0.11$ & $0.68$ & $0.49$ & $0.14$ & $0.84$ & $101.2$s\\
    \hline
    \hline
    \multirow{3}{*}{$p=6$}& Predictability & $0.13$ &  $0.50$ & $0.31$ & $0.19$ & $0.07$ & $131.0$s\\
    \cline{2-8}
       &Suboptimality & $0.11$ & $0.65$ & $0.46$ & $0.14$ & $0.72$ & $121.9$s\\
    \hline
\end{tabular}
\end{center}}
\vskip -0.1in
\end{table}

\begin{table}[h!]
\caption{Noisy case with $200$ training samples.}\label{tab:multicol:bilinear:noise200}
\begin{center}
\small{
\begin{tabular}{|c|c|c|c|c|c|c|c|}
    \hline
   $n=5$ & $N=200$ & train &$\ell_\theta^p$ & $\ell^p$ &  $\ell_\theta^\text{sub}$ &  $\ell^\text{sub}$ & time \\
    \hline
    \multirow{3}{*}{$p=4$}& Predictability & $0.52$ & $1.17$ & $0.97$ & $0.58$ & $0.10$ & $170.4$s \\
       \cline{2-8}
       & Suboptimality & $0.44$ & $1.46$ & $1.26$ & $0.56$ & $0.97$ & $208.3s$\\
    \hline
    \hline
    \multirow{3}{*}{$p=5$}& Predictability & $0.16$ &  $0.37$ & $0.18$ & $0.17$ & $0.05$ & $259.2$s \\
       \cline{2-8}
       & Suboptimality & $0.11$ & $0.78$ & $0.58$ & $0.14$ & $0.89$ & $222.4$s\\
    \hline
    \hline
    \multirow{3}{*}{$p=6$}& Predictability & $0.14$ &  $0.37$ & $0.19$ & $0.16$ & $0.10$ & $257.2$s \\
    \cline{2-8}
       &Suboptimality & $0.11$ & $0.74$ & $0.54$ & $0.13$ & $0.84$ & $236.2$s\\
    \hline
\end{tabular}}
\end{center}
\vskip -0.1in
\end{table}

\textbf{Results.} The results are summarized in Table \ref{tab:multicol:bilinear:noiseless}, \ref{tab:multicol:bilinear:noise100}, and \ref{tab:multicol:bilinear:noise200}. In scenarios without noise, the forward problem is accurately restored when the hypothesis class encompasses the true feasible region. In such instances, this requirement is equivalent to having $p \geq 5$ since the original feasible region comprises five extreme points. In noisy scenarios, we noticed enhanced performance with an enlarged sample size for predictability loss, with no corresponding improvement observed for suboptimality loss. It was also noted that the MILP reformulation exhibited superior overall performance in contrast to the gradient-descent-based algorithm. Nonetheless, in instances characterized by noise, the inherent complexity of MILP formulations renders them challenging (number of binary variables proportional to data size) to resolve within a one-hour timeframe. Finally, a visual representation of feasible region  $g_{\theta^t}(\boldsymbol x, \boldsymbol s)\leq 0$ as the parameters $\theta^t$ are updated using  Algorithm~\ref{Alg:smooth} is presented in Figure~\ref{fig:enter-label} for the noiseless case. We observe that the hypothesis class is able to rotate, scale, project and translate the primitive set by updating $\theta^t = (A^t,b^t)$.  
\begin{figure}[h!]
    \centering
    \includegraphics[width=0.6\linewidth]{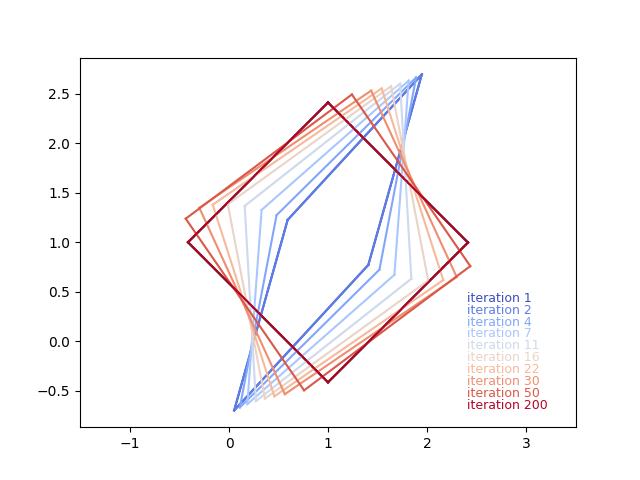}
    \caption{Visualization of $g_{\theta^t}(\boldsymbol x, \boldsymbol s)\leq 0$ for Algorithm~\ref{Alg:smooth}. The feasible region of the forward problem is given by $\|\boldsymbol x - \boldsymbol e\|\leq 1$, and the learning algorithm recovers it by iteration 200.}
    \label{fig:enter-label}
\end{figure}

\end{document}